\documentclass{amsart}
\usepackage{amsmath,amssymb,amsthm,amscd, amsxtra, amsfonts}
\begin{document}

\newtheorem{theorem}{Theorem}[section]
\newtheorem{lemma}[theorem]{Lemma}
\newtheorem{definition}[theorem]{Definition}
\newtheorem{example}[theorem]{Example}
\newtheorem{proposition}[theorem]{Proposition}
\newtheorem{corollary}[theorem]{Corollary}
\newtheorem{conjecture}[theorem]{Conjecture}
\newtheorem{remark}{Remark}
\newcommand{\wta}{{\rm {wt} }  a }
\newcommand{\R}{\frak R}
\def \<{\langle}
\def\LL{\mathcal{L}}
\def \>{\rangle}
\def \t{\tau }
\def \a{\alpha }
\def \e{\epsilon }
\def \l{\lambda }
\def \ga{\gamma }
\def \L{ L}
\def \b{\beta }
\def \om{\omega }
\def \o{\omega }
\def \c{\chi}
\def \ch{\chi}
\def \cg{\chi_g}
\def \ag{\alpha_g}
\def \ah{\alpha_h}
\def \ph{\psi_h}
\def \gi {\gamma}
\def\supp{{\rm{supp}}}
\def\GG{\mathcal{G}}
\def\NN{\mathcal{N}}
\newcommand{\wtb}{{\rm {wt} }  b }
\newcommand{\bea}{\begin{eqnarray}}
\newcommand{\eea}{\end{eqnarray}}
\newcommand{\be}{\begin {equation}}
\newcommand{\ee}{\end{equation}}
\newcommand{\g}{\frak g}
\newcommand{\tg}{\tilde {\frak g} }
\newcommand{\hg}{\hat {\frak g} }
\newcommand{\hb}{\hat {\frak b} }
\newcommand{\hn}{\hat {\frak n} }
\newcommand{\h}{\frak h}
\newcommand{\wt}{{\rm {wt} }   }
\newcommand{\V}{\cal V}
\newcommand{\hh}{\hat {\frak h} }
\newcommand{\n}{\frak n}
\newcommand{\Z}{\mathbb Z}
\newcommand{\lar}{\longrightarrow}
\newcommand{\X}{\mathfrak{X}}

\newcommand{\Zp}{{\Bbb Z}_{\ge 0} }

\newcommand{\N}{\mathbb{N}}
\newcommand{\C}{\mathbb C}
\newcommand{\Q}{\Bbb Q}
\newcommand{\WW}{\boldsymbol{ \mathcal{W}}}
\def\gl{\mathfrak{gl}}
\def\sl{\mathfrak{sl}}
\newcommand{\la}{\langle}
\newcommand{\ra}{\rangle}
\newcommand{\bb}{\mathfrak{b}}
\newcommand{\triplet}{\mathcal{W}(p)}
\newcommand{\striplet}{\mathcal{SW}(m)}

\newcommand{\nordplus}{\mbox{\scriptsize ${+ \atop +}$}}
\newcommand{\ds}{\displaystyle}
\newcommand{\halmos}{\rule{1ex}{1.4ex}}
\newcommand{\pfbox}{\hspace*{\fill}\mbox{$\halmos$}}
\newcommand{\epfv}{\hspace*{\fill}\mbox{$\halmos$}}
\newcommand{\nb}{{\mbox {\tiny{ ${\bullet \atop \bullet}$}}}}
\newcommand{\nn}{\nonumber \\}

\newcommand{\HH}{\widetilde{\mathcal{T}}}
\newcommand{\vak}{\bf 1}

\def\sect#1{\section{#1}\setcounter{equation}{0}\setcounter{rema}{0}}
\def\ssect#1{\subsection{#1}}
\def\sssect#1{\subsubsection{#1}}

\title[]{
  Whittaker modules for  the affine Lie algebra $A_1 ^{(1)}$   }


\author{ Dra\v zen Adamovi\' c }

\curraddr{D.A. Department of Mathematics, University of Zagreb,
Bijeni\v cka 30, 10 000 Zagreb, Croatia} \email {adamovic@math.hr}

\author{Rencai L\"u}
\curraddr{R.L. Department of Mathematics, Soochow
University,  Suzhou,  P. R. China}
\email{rencail@amss.ac.cn}

\author{Kaiming Zhao}
\curraddr{K.Z. Department of Mathematics, Wilfrid Laurier
University, Waterloo, Ontario, N2L 3C5, Canada; and College of
Mathematics and Information Science, Hebei Normal (Teachers)
University, Shijiazhuang 050016, Hebei, P. R. China }
\email{kzhao@wlu.ca}

\markboth{Dra\v zen Adamovi\' c, Rencai L\"u and Kaiming Zhao } { Whittaker modules for $A_1 ^{(1)} $}
\bibliographystyle{amsalpha}
\pagestyle{myheadings}

\begin{abstract}
We prove the irreducibility of the  universal non-degenerate Whittaker modules for
the affine Lie algebra $\widehat{\sl_2}$ of type $A_1^{(1)}$ with noncritical level. These modules can become simple Whittaker modules over $\widetilde{\sl_2} =\widehat{\sl_2} + {\C} d $ with the same Whittaker function and central charge. We have to modulo a central character for ${\sl_2}$ to obtain
simple degenerate Whittaker $\widehat{\sl_2} $-modules with noncritical level.
In the case of critical level the universal Whittaker module is
reducible.  We prove that the quotient of universal
Whittaker $\widehat{\sl_2}$--module by a submodule generated by a scalar action of central
elements of the  vertex algebra $V_{-2}(\sl_2)$ is simple as $\widehat{\sl_2}$--module. We also explicitly describe the simple quotients of universal Whittaker modules at the critical level for $\widetilde{\sl_2}$. Quite surprisingly, with the same Whittaker function some simple degenerate $\widetilde{\sl_2}$ Whittaker modules at the critical level   have semisimple action of $d$ and others have  free action of $d$.
At last, by using vertex algebraic techniques we  present a Wakimoto type
construction of a family of  simple generalized Whittaker  modules for $\widehat{\sl_2}$ at the critical level. This family includes all classical Whittaker modules at critical level. We also have Wakimoto type realization for degenerate Whittaker  modules for $\widehat{\sl_2}$ at noncritical level.
\end{abstract}

\maketitle

\tableofcontents

\section{Introduction}

Recently, Whittaker modules over various Lie algebras are attracting a lot of
attention from mathematicians. See \cite{MZ2} and the references
there.
 In Block's classification of all simple modules for the
three-dimensional simple Lie algebra $\sl_2$, they fall into two
families: highest (lowest) weight modules,  and a family which are simple modules over a Borel subalgebra
of $\sl_2$ including Whittaker modules (see \cite{B}). Whittaker modules are an important class
of simple modules. Kostant defined and systematically studied
in  \cite{Ko} Whittaker modules for an arbitrary finite dimensional
complex semisimple Lie algebra $\g$. He showed that these modules
with a fixed regular Whittaker function (Lie homomorphism) on a nilpotent
radical are (up to isomorphism) in bijective correspondence with
central characters of $U(\g)$.

McDowell  studied in  \cite{Mc} a category of modules for an
arbitrary finite-dimensional complex semisimple Lie algebra $\g$
which includes the Bernstein-Gelfand-Gelfand category $\mathcal {O}$
as well as those Whittaker modules where the Whittaker  function   on a nilpotent radical may be irregular (degenerate). The
simple objects in this category are constructed by inducing
over a parabolic subalgebra $\mathfrak{p}$ of $\g$ from a simple
Whittaker module (in Kostant's sense) or from a highest weight
module for the reductive Levi factor of $\mathfrak{p}$ (when the Whittaker function is zero).

The study on Whittaker modules over the Virasoro algebra was originated  in \cite{OW1, OW2}. Complete and more general results were obtained in \cite{GLZ, MZ2}. The paper \cite{LZ} completely investigated Whittaker modules over the  Heisenberg-Virasoro algebra.

Affine Lie algebras are the most extensively studied and most useful
ones among infinite-dimensional Kac-Moody  algebras.  The integrable highest weight modules were the first class of
representations over affine Kac-Moody  algebras being extensively studied,
see \cite{K} for detailed discussion of results and further bibliography.
In \cite{Ch}, Chari classified all simple integrable weight modules
with finite-dimensional weight spaces over the untwisted affine Lie
algebras. Chari and Pressley \cite{CP2}, then extended this
classification to all affine Lie algebras.  Verma-type modules
were first studied by Jakobsen and Kac \cite{JK}, and then by Futorny
\cite{Fu1,Fu2}.

Very recently, a complete classification for all simple weight
modules with finite-dimensional weight spaces over affine Lie
algebras were obtained in \cite{FT, DG}.  Naturally, the next important task
is to study simple weight modules with infinite-dimensional
weight spaces and simple non-weight modules.  Besides the simple modules constructed
in \cite{CP1}, a  class of
simple weight modules over affine Lie algebras with
infinite-dimensional weight spaces were constructed  in \cite{BBFK}.
A complete classification for all simple (weight and
non--weight) modules over affine Lie algebras with locally nilpotent
action of the nilpotent radical were obtained in \cite{MZ}. All  simple (weight and
non-weight) modules over untwisted affine Lie algebras with locally finite
action of the nilpotent radical were classified in \cite{GZ}, where the structure of simple Whittaker modules were unclear.

A class of simple non--weight  modules for untwisted affine Lie algebras
from simple Whittaker modules over the subalgebra generated by
imaginary root spaces (isomorphic to an infinite dimensional
Heisenberg algebra) were constructed in \cite{Ch1}. These modules are
called imaginary Whittaker modules since they are different from the
above Whittaker modules in nature. We should mention that imaginary Whittaker modules from \cite{Ch1} are not modules for affine vertex algebras.

 It is natural to investigate  Whittaker modules over  affine Kac-Moody algebras in
the more general setting as in \cite{Mc}, i.e.,   the Whittaker function  on the nilpotent radical may be irregular.  These Whittaker modules can be considered as modules of affine vertex algebras of a suitable level. To distinguish from imaginary Whittaker modules we sometimes call these Whittaker modules as {\it classical Whittaker modules}. The
objective of the present paper is to completely determine the structure of classical Whittaker modules over the
 affine Kac-Moody algebra $A_1 ^{(1)}$ including $\widehat{\sl_2}$ (the universal central extension of the loop algebra of $\sl_2$) and $\widetilde{\sl_2}=\widehat{\sl_2}+\C d $ (where $d$ is the degree operator).
See Section \ref{preliminaries} for the notations.

 Let us here describe the main results of the paper. Let $(\lambda, \mu ) \in {\C} ^2$, and $V_{\widehat{\sl_2} } (\lambda, \mu, \kappa)$ be the universal Whittaker module of level $\kappa$ generated by vector $w_{\lambda, \mu, \kappa}$ such that $e(0) w_{\lambda, \mu, \kappa} = \lambda w_{\lambda, \mu, \kappa} $, $f(1) w_{\lambda, \mu, \kappa} = \mu w_{\lambda, \mu, \kappa}$
 (for details see the last part of Section \ref{whittaker}).
 If $\lambda \cdot \mu \ne 0$ we call such Whittaker modules {\it non-degenerate}. We describe all the simple Whittaker modules for $\widehat{\sl_2}$ and $\widetilde{\sl_2}$. (For notations, see Section \ref{whittaker}.)

 \begin{theorem}
 Assume that $\lambda \in {\C} ^*$, $\mu, \kappa \in {\C}$.
\begin{itemize}
 \item[(i)] Let $\mu \ne 0$. Then $V_{\widehat{\sl_2} } (\lambda, \mu, \kappa)$ is a simple $\widehat{\sl_2}$--module  at level  $\kappa \ne -2$.
 \item[(i')] Let $\mu \ne 0$ and  $d= -L(0) + a$, $a \in {\C}$. Then  $V_{\widehat{\sl_2} } (\lambda, \mu, \kappa)$  is an simple $\widetilde{\sl_2}$--module at level ${\kappa}\ne-2$.
 \item[(ii)] For every  $c (z) = \sum_{n \le  0} c _n z ^{-n-2} \in {\C}((z))$,
 $$ V_{\widehat{\sl_2} } (\lambda, \mu, -2 , c (z) ) = V_{\widehat{\sl_2} } (\lambda, \mu, -2) / \langle (T(n) - c_n ) w_{\lambda, \mu, -2}, \ n \le 0 \rangle $$
 is  a simple $\widehat{\sl_2}$--module at the critical level.
 \item[(ii')]Let $\mu \ne 0$. The induced module
 $${\rm Ind}_{\widehat{\sl_2}}  ^{\widetilde{\sl_2}} V_{\widehat{\sl_2} } (\lambda, \mu, -2, c (z) )$$
 is simple Whittaker module at the critical level.\end{itemize}
 Modules described above provide a complete list of non-degenerate simple Whittaker modules for  $\widehat{\sl_2}$ and $\widetilde{\sl_2}$.
 \end{theorem}

This theorem has a long proof which will be broken into many cases, see Remark \ref{remark-proof}   in Section \ref{section-whittaker-widetilde-sl2}. In the proof of these results we show that simple Whittaker modules for $\widehat{\sl_2}$ are simple modules for the Lie algebra $\hat{\frak b} \rtimes {\rm{Vir}}$ (see Section \ref{def-borel-virasoro}) in the non-critical case and for the Lie algebra $\hat{\frak b} \rtimes \HH $ (see Section \ref{def-borel-virasoro-critical})  in the case of critical level.

The only case which is not completely described in the above theorem for $\widehat{\sl_2}$ is the case of degenerate Whittaker modules outside the critical level.
We prove in Theorem \ref{ired-non-crit} that simple degenerate Whittaker modules of type $(\lambda,0)$ with noncritical level (${\kappa}\ne-2$) are simple quotients of the $\widehat{\sl_2}$--module
$$  M_{\widehat{\sl_2} } (\lambda, 0 , \kappa, a ):= V_{\widehat{\sl_2} } (\lambda, 0 , \kappa) /   U ({\widehat{\sl_2} }) (L(0)- a) w_{\lambda, 0, \kappa}, \quad (a \in {\C}). $$

Irreducible degenerate  Whittaker $\widetilde{\sl_2}$-modules with critical level are quite complicated which are completely determined in Theorem \ref{thm-critical-case2}.
Surprisingly, with the same Whittaker function,  simple $\widetilde{\sl_2}$ Whittaker modules at critical level can have semisimple or free action of $d$.

The second part (Sections \ref{section-Wakimoto-1} - \ref{realization-d}) of our paper is devoted to the explicit  realizations of simple Whittaker modules over $\widehat{\sl_2}$ and $\widetilde{\sl_2}$. We use the concept of Wakimoto modules for $\widehat{\sl_2}$  and the theory of vertex algebras. At the critical level we present explicit realization of a family of  simple Whittaker modules which include all classical (degenerate and non-degenerate) simple Whittaker modules. Let us present this result in more details.

Let $M$ be a Weyl vertex algebra generated by the fields $a(z)=\sum_{n \in {\Z} } a(n) z ^{-n-1}$, $a^*(z) = \sum_{n \in {\Z} } a^*(n) z ^{-n-1}$ such that the components of these fields satisfies the commutation relations (\ref{comut-Weyl}) in Section \ref{Weyl} for the infinite dimensional Weyl algebra. Let $\pi ^{2 (\kappa+2)}$ be Heisenberg vertex algebra associated to the representations of the Heisenberg Lie algebra of level $ 2(\kappa+2)$.
 Construction of Wakimoto modules is based on the embedding of the affine vertex algebra $V_\kappa(\sl_2)$ into $M \otimes \pi ^{2 (\kappa+2)}$ (cf. \cite{efren}, \cite{W-mod}). This implies that for any  $M$--module $M_1$ and any $\pi ^{2 (\kappa+2)}$--module $N_1$ the tensor product $M_1 \otimes N_1$ is a module for  $V_\kappa(\sl_2)$. This construction was usually applied on highest weight modules for $M$ and  $\pi ^{2 (\kappa+2)}$. In the present paper we shall apply this construction of Whittaker modules for   $M$ and  $\pi ^{2 (\kappa+2)}$.

For $(\lambda, \mu) \in {\C} ^2$ let $M_1(\lambda, \mu)$ be the module for the Weyl algebra generated by the Whittaker vector $v_1$ such that
$$a(0) v_1 = \lambda v_1, \ a^* (1) v_1 = \mu v_1, \ a(n+1) v_1 = a^* (n+2) v_1 = 0 \quad (n \ge 0). $$

By using Whittaker module $M_1(\lambda, \mu)$ and certain Whittaker modules for the Heisenberg vertex algebra we construct a family of modules on arbitrary level
(see Section \ref{whitt-wak}). In the case of critical level, $\pi^0$ is a commutative vertex algebra and their simple (Whittaker) representations are $1$--dimensional, so our Wakimoto modules will be actually realized on $M_1(\lambda, \mu)$.
We prove in Theorem \ref{ired-wak-1}:

\begin{theorem}
For every $\chi(z) \in {\C}((z))$, $(\lambda, \mu) \in {\C} ^2$, $\lambda \ne 0$  there exists simple $\widehat{\sl_2}$--module $\overline{M_{Wak}} (\lambda, \mu, -2, \chi(z) )$ realized on the $M$--module $M_1(\lambda, \mu)$ such that
\bea
e(z)   &=& a(z); \nonumber \\
h(z) &=& -2 : a^{*}(z)  a (z) : + \chi(z) ; \nonumber \\
f(z)   &=& - : a^{*} (z) ^2 a (z) : -2 \partial_z a^{*} (z) + a^{*}(z) \chi (z) \nonumber
\eea
\end{theorem}
In the case when $\mu = 0$  and suitable $\chi(z)$ the Wakimoto modules above provide a complete list of simple Whittaker modules of type $(\lambda, 0)$.
But it is interesting that Wakimoto modules don't present a construction of classical non-degenerate Whittaker modules of type $(\lambda, \mu)$.

In order to present a realization of non-degenerate simple Whittaker modules at the critical level  we slightly generalize the concept of Wakimoto modules. We consider the vertex algebra
$\Pi (0)$  of lattice type  which can be treated as a localization of the Weyl vertex algebra. So in $\Pi(0)$ we have the field $a^{-1} (z)$ such that
$$ a(z) a^{-1} (z) =Id. $$

We construct an embedding  of  $V_{-2}(\sl_2)$ into a vertex algebra $M_T(0) \otimes \Pi(0)$ (see Proposition  \ref{embedding-new}) which enables us to construct classical non-degenerate Whittaker modules. Then for every $ \lambda \in {\C} $ and $\chi(z) \in {\C}((z))$ we consider Whittaker modules $\Pi_{\lambda}$ for $\Pi(0)$
 and  $1$--dimensional representation $M_T( \chi(z) )$  of  $M_T(0)$ such that $T(n)$ acts as $\chi_n \in {\C}$.

We prove in Theorem \ref{wak-nondeg-tm}:

\begin{theorem} Let $\lambda    \in {\C}^*$, $\mu \in {\C}$  and $c(z) =  \sum_{n \le 0 }  c (n) z ^{-n-2} \in {\C}((z))$. Set
$ \chi (z) = \frac{\lambda \mu }{ z^3} +  c(z) $.
Then we have:
$$V_{\widehat{\sl_2} } (\lambda, \mu, -2, \chi(z) ) \cong M_T( \chi(z) ) \otimes \Pi_{\lambda  }. $$
\end{theorem}

In our forthcoming papers we plan to generalize  our results on higher level affine Lie algebras.

\section{Preliminaries}
\label{preliminaries}
\subsection{Affine vertex algebra $V_{\kappa} (\sl_2) $ and its modules}
Let us first recall the related affine Lie algebra theory and vertex   algebra theory from \cite{FB, K2, LL, DL}.

Let ${\g}$ be a finite-dimensional simple Lie algebra over ${\Bbb
C}$ and let $(\cdot,\cdot)$ be a nondegenerate symmetric bilinear
form on ${\g}$. Let  ${\g} = {\n}_- + {\h} + {\n}_+$    be a
triangular decomposition for ${\g}$.

 The affine Lie algebra ${\tg}$ associated
with ${\g}$ is defined as $ {\g} \otimes {\C}[t,t^{-1}] \oplus
{\C}c \oplus {\C}d, $ where $c$ is the canonical central element
\cite{K}
 and  the Lie algebra structure
is given by \bea \label{comut-af-1} [ x \otimes t^n, y \otimes t^m] &=& [x,y] \otimes t
^{n+m} + n (x,y) \delta_{n+m,0} c, \\ \label{comut-af-2} [d, x \otimes t^n] &=& n x
\otimes t^n \eea for $x,y \in {\g}$ and $m,n\in \mathbb{Z}$.  We will
write $x(n)$ for $x \otimes t^{n}$.

The Cartan subalgebra ${\tg}_0$ and  subalgebras ${\tg}_+$,
${\tg}_-$ of ${\tg}$ are defined by $${\tg}_0 = {\h} \oplus {\C}c
\oplus {\C}d, \quad
 {{\tg}}_{\pm} =  \g \otimes t^{\pm} {\C}[t^{\pm}]+{\n}_{\pm}\otimes \C.$$
This defines a triangular decomposition
$$ \tg=\tg_+\oplus \tg_0\oplus \tg_-. $$

  Denote $\hg=[\tg,\tg]=\hg_+\oplus \hg_0\oplus \hg_-$, where $\hg_0={\h} \oplus {\C}c, \hg_{\pm}=\tg_{\pm}$.

Let  $P = {\g}\otimes {\C}[t] \oplus {\C}c $.
  For every $\kappa \in {\C}$,  let ${\C} v_{\kappa}$ be
$1$-dimensional $P$--module  such that the subalgebra ${\g}
\otimes  {\C}[t]$ acts trivially,
 and  the central element
$c$ acts as multiplication with ${\kappa}$. Define the
generalized Verma module $V_{\kappa}(\g)$ as
$$V_{\kappa} (\g)  = U(\tg) \otimes _{ U(P) } {\C} v_{\kappa} .$$
Then $V_{\kappa}(\g)$ has a natural structure of a vertex
algebra generated by fields
$$ x (z) = Y( x(-1) {\bf 1}, z) = \sum_{n \in {\Z} } x(n) z ^{-n-1}, \quad (x \in {\g} ),$$
where  ${\bf 1} = 1 \otimes v_{\kappa}$ is the vacuum vector.

A $\hg$--module $N$ is called {\it restricted} if  for every $w \in N$ and $x \in \g$ we have
$$ x(z) w \in {\C} ((z)). $$
 A restricted $\hg$--module of level $\kappa$ has the structure of a module over vertex algebra $V_{\kappa}(\g)$.

  From now on in this paper let $\g=\sl_2$ with standard generators $e,f, h$; and ${\frak b}=\C h+\C e$. Note that $$[e,f]=h, \ [h, e]=2e, \ [h,f]=-2f, \ (e,f)=1\text{ and }(h,h)=2.$$ We see that $\h=\C h$, $\n_+=\C e$ and $\n_-=\C f$. Let $\hb$ be the subalgebra of $\hg$ generated by $e(n), h(n)$ for $n\in\Z$.

 Assume first that $\kappa \ne -2$.
 Let
$$ \omega = \frac{1}{2 (\kappa+2)} \left( e(-1) f(-1) + f(-1) e(-1) + 1/2 h(-1) ^2 \right) {\bf 1} $$ be the canonical Sugawara Virasoro vector in the vertex algebra $V_{\kappa}(\sl_2)$. Then the components of the field
$$Y(\omega, z) = L(z) = \sum_{n \in {\Z} } L(n) z ^{-n-2} $$
satisfies the commutation relation for the Virasoro algebra of central charge $c_\kappa = \frac{3\kappa }{\kappa+2}$.
Then every module over $V_{\kappa} (\sl_2)$ becomes  a module for the Virasoro algebra. Recall also that
\bea \label{aff-vir}   [L(n), x (m) ] = - m x(n+m) \quad \mbox{for} \ x \in \{ e,f, h \}.  \eea
In particular, $$ [L(n), x(0) ] = 0 \quad \mbox{for} \ x \in \{ e,f, h \}. $$

\vskip 5mm
Let $\kappa=-2$.
Let $t=\frac 1 2 (e(-1) f(-1) + f(-1) e(-1)+ \frac{1}{2} h(-1) ^2 ) {\bf 1} \in V_{-2}(\sl_2)$ and
$$ T(z) = Y(t,z) = \sum_{n \in {\Z} } T(n) z ^{-n-2}. $$
Then
$$ [T(n), x (m) ] = 0 \quad \forall m, n \in {\Z}, $$
i.e., $T(n)$ are central elements. In particular, $t$ generates the center of the vertex algebra $V_{-2}(\sl_2)$ (cf. \cite{efren}, \cite{FB}).
The center is a commutative vertex algebra $M_T (0)$ which is as vector spaces isomorphic to the polynomial algebra
${\C} [T(-n) \vert \  n \ge 0]$.

\subsection{Lie algebra $\widehat{\frak b} \rtimes {\text{Vir}} $}
\label{def-borel-virasoro}

Let $$\LL= \widehat{\frak b} \rtimes \text{Vir} =\text{span}\{L(i),h(i),e(i),c_1,c|i\in \Z\}$$ be the Lie algebra with the Lie brackets
\bea
&& [L(i),L(j)]=(i-j)L(i+j)+\frac{i^3-i}{12}c_1, \nonumber \\
&& [L(i),h(j)]=-jh(i+j),\,\,\,\,[L(i),e(j)]=-je(i+j), \nonumber \\
&& [h(i),h(j)]=\delta_{i+j,0}2ic,\,\,\,\,[h(i),e(j)]=2e(i+j),  \nonumber \\
&&[e(i),e(j)]=0, \nonumber \\
&& [c_1,\LL]=[c,\LL]=0. \nonumber \eea

 Now we can establish
the following  connection between $V_{\kappa}(\g)$--modules  and  $\LL$--modules.

\begin{proposition} \label{ired-crit-1}  Let $\kappa\ne -2$, and let  $N$ be any  $V_{\kappa}(\g)$--module. Then  the following statements hold: \itemize
\item[(i)] The module $N$ is also an $\LL$--module such that $c = \kappa \mbox{Id} $ and $c_1 =  \frac{3 \kappa}{\kappa +2} \mbox{Id}$.
\item[(ii)]  If $N$ is a simple $\LL$--module, then $N$ is a simple $V_{\kappa} (\g)$--module.\enditemize
\end{proposition}
\begin{proof}
By using  the Sugawara construction above and relation (\ref{aff-vir}) we get assertion (i). If $N$ is a simple $\LL$--module, then $N$ is a simple module for the vertex subalgebra of $ V_{\kappa} (\g)$ generated by $\frak b$ and the Virasoro vector $\omega$. Therefore, $N$ is also a simple $V_{\kappa} (\g)$--module.
\end{proof}

\subsection{Lie algebra $\widehat{\frak b} \rtimes {\HH} $}
 \label{def-borel-virasoro-critical}
In our paper we shall study graded representations of the graded center of $V_{-2}(\sl_2)$. This leads to the study of the following infinite-dimensional Lie algebra.

Let $\HH$ be the Lie algebra with the basis $\{d, T(n)|n\in \Z\}$ and the Lie bracket $$[T(i),T(j)]=0, \,\,[d, T(i)]=iT(i),\forall i,j\in \Z.$$
Let ${\HH}_{-}$ denote its subalgebra $\{d, T(n)|n\le 0 \}$.
The Lie algebra $\hat{{\frak b}} \rtimes\HH$ is defined by
$$[d,x(n)]=nx(n), [c, d]=[T(n),c]=[T(n),x(n)]=0, \forall n\in \Z, x\in\{e, h\}.$$
Similarly we have Lie algebra  $\hat{{\frak b}} \rtimes{\HH} _{-}$

We have the following useful criterion.

\begin{proposition} \label{ired-crit-2}  Let  $N$ be any  restricted  $\widetilde{sl_2}$--module at the critical level. Then  the following statements hold:\itemize
\item[(i)] The module $N$ is also an $\widehat{\frak b} \rtimes {\HH} $--module.

\item[(ii)]  If $N$ is a simple $\widehat{\frak b} \rtimes {\HH} $--module, then $N$ is a simple $\widetilde{sl_2}$--module.\enditemize
\end{proposition}
\begin{proof}
The proof of (i)  easily follows from the fact that $N$ is also a module for the center of $V_{-2}(\g)$ and the commutation relation (\ref{comut-af-2}).  The proof of the assertion (ii) is similar as in the proof of Proposition \ref{ired-crit-1}.
\end{proof}
We shall  also  present a method for constructing  irreducible $\widehat{\frak b} \rtimes {\HH} $--modules which we shall use in the paper. The proof of the following result  easily follows from Theorem 7 of \cite{LZ}.

 \begin{proposition} \label{constr-bt-modules}
 Assume that $Y_1, Y_2$ are $\widehat{\frak b} \rtimes {\HH} $--modules such that $Y_2$ is a simple module for $\widehat{\frak b}$  and $\widehat{\frak b}$ acts trivially on $Y_1$. Then the following holds:
 \itemize\item[(1)] Any submodule of $Y_1\otimes Y_2$ is of the form $Y'_1 \otimes Y_2$ for certain submodule $Y'_1$ of $Y_1$.
 \item[(2)] If $Y_1$ is  a simple ${\HH}$--module , then $Y_1 \otimes Y_2$ is a simple $ \widehat{\frak b} \rtimes {\HH} $--module.
\enditemize
 \end{proposition}

\subsection{Weyl vertex algebra}
\label{Weyl}

Recall that the  {\it Weyl algebra} is an associative algebra with generators
$$ a(n), a^{*} (n) \quad (n \in {\Z})$$ and relations
\bea  \label{comut-Weyl} && [a(n), a^{*} (m)] = \delta_{n+m,0}, \quad [a(n), a(m)] = [a ^{*} (m), a ^* (n) ] = 0 \quad (n,m \in {\Z}). \eea
Let $M$ denotes the simple {\it Weyl module} generated by the cyclic vector ${\bf 1}$ such that
$$ a(n) {\bf 1} = a  ^* (n+1) {\bf 1} = 0 \quad (n \ge 0). $$
As a vector space $$ M \cong {\C}[a(-n), a^*(-m) \ \vert \ n < 0, \ m \le 0 ]. $$

\vskip 5mm

Let us describe the basis of $M$. Recall that a partition is a sequence of non-negative integers
$$ (\lambda_1, \lambda_2, \cdots )$$
such that
$$ \lambda_1 \ge \lambda_2 \ge \cdots \quad \mbox{and} \quad \lambda_n = 0 \quad \mbox{for} \ n \ \mbox{sufficiently large}. $$
Let $\mathcal{P}$ be the set of all partitions. Define the {\it length}  $\ell(\lambda)$ and {\it size} $\vert \lambda \vert$ of partitions by
$$\ell(\lambda) = \mbox{max}\{ n \ \vert \ \lambda_n \ne 0 \}, \quad \vert \lambda \vert = \lambda_1 + \lambda_2 + \cdots   .$$
If $\ell(\lambda) = \ell$ we write $\lambda = (\lambda_1, \dots, \lambda_{\ell})$.
Let $\mathcal{P}_{\ell}$ be set of all partitions of length $\ell$. Let $\phi$ be the partition with all the entries being zero. Then we write $\ell(\phi) = 0$.

We shall also need a total order on the set $\mathcal{P}_{\ell}$ such that
$$ \lambda > \mu \quad \mbox{if} \ \lambda_1 = \mu_1, \cdots, \lambda_{i-1} = \mu_{i-1} \quad \mbox{and} \ \lambda_i > \mu_i \quad \mbox{for some} \ i. $$

For  partitions $\lambda = (\lambda_1, \dots, \lambda_r)$,  $\mu = (\mu_1, \dots, \mu_s)$ we set
$$ u_{\lambda, \mu} :=a (-\lambda_1) \cdots a(-\lambda_r) a^* (-\mu_1 +1) \cdots a^* (-\mu_s+1), $$
$$ u_{\lambda, \phi}:= a (-\lambda_1) \cdots a(-\lambda_r), \quad u_{\phi, \mu}:= a^* (-\mu_1 +1) \cdots a^* (-\mu_s+ 1), $$
$$ u _{\phi, \phi} = 1. $$
Then the set
$$\{ u_{\lambda, \mu} {\bf 1} \ \vert (\lambda, \mu) \in \mathcal{P} \times \mathcal{P} \}$$
is a basis for $M$.

\vskip 5mm
 There is a unique vertex algebra $(M, Y, {\bf 1})$  where
the  vertex operator map is $$ Y: M \rightarrow \mbox{End}[[z, z ^{-1}]] $$
such that
$$ Y (a(-1) {\bf 1}, z) = a(z), \quad Y(a^* (0) {\bf 1}, z) = a ^* (z),$$
$$ a(z)   = \sum_{n \in {\Z} } a(n) z^{-n-1}, \ \ a^{*}(z) =  \sum_{n \in {\Z} } a^{*}(n)
z^{-n}. $$

In particular we have:
$$ Y (a(-1) a^*  (0) {\bf 1}, z) =  a(z) ^+ a^* (z)  +  a ^* (z)  a(z) ^- ; $$
$$ Y( a(-1) a ^* (0) ^2 {\bf 1}, z) = a(z) ^+ (a^* (z) ) ^2 + ( a ^* (z) ) ^2 a(z) ^- $$
where
$$a (z) ^+ = \sum_{n \le -1} a(n) z ^{-n-1}, \quad a (z) ^- = \sum_{n \ge  0} a(n) z ^{-n-1}. $$

\section{Whittaker modules}
\label{whittaker}

In this section we will recall Whittaker modules over a Lie algebra with a triangular decomposition and establish  related results.

\subsection{Some general results and definitions} Let $\GG$ be a complex Lie algebra with triangular decomposition
$\GG = \GG_+ \oplus \GG_0 \oplus \GG_-$ in the sense of \cite{MP}.
Let $\eta: \GG_+ \rightarrow \C$ be a Lie algebra homomorphism which will be called a {\it  Whittaker function}, and
$V$ be a $\GG$--module. Note that $\eta([\GG_+ ,\GG_+ ])=0$. A nonzero vector $v \in V$ is called a {\it
Whittaker vector of type $\eta$} if $xv=\eta(x)v$ for all $x \in
\GG_+$. The module $V$ is said to be a {\it type $\eta$ Whittaker module for
$\GG$} if it is generated by a type $\eta$ Whittaker vector. We say that $\GG_+$ acts on $V$ {\it locally nilpotently} if for any $v\in V$ there is $s\in\N$ depending on $v$ such that $x_1x_2...x_sv=0$ for any $x_1,x_2,...,x_s\in\GG_+$.
Let $\GG_+^{(\eta)}=\{x-\eta(x)\,|\,x\in\GG_+\}$ which is a Lie subalgebra of the universal enveloping algebra $U(\GG_+)$ of $\GG_+$.

The following result is somewhat known for  some Lie algebras, see \cite{B, GLZ, LZ}.

\begin{lemma}\label{lemma2.1} Let $V$ be a type $\eta$ Whittaker module for $\GG$.
Suppose that $\GG_+$ acts locally nilpotently on $\GG/\GG_+$. Then
\begin{itemize}
 \item[(i)]$\GG_+^{(\eta)}$  acts locally nilpotently on $V$. In particular, $x-\eta(x)$ acts locally nilpotently on $V$ for any $x\in\GG_+$;
\item[(ii)] any nonzero submodule of $V$ contains a Whittaker vector of type
$\eta$;
\item[(iii)] if the vector space of Whittaker vectors of $V$ is $1$-dimensional, then $V$ is simple.
\end{itemize}
\end{lemma}
\begin{proof} Let $V=U(\GG)v$ where $v$ is  a type $\eta$ Whittaker vector.

(i)  Let $w=y_1y_2...y_rv$
 be a nonzero vector in $V$ where $y_i \in  \GG$. It is enough to show that $\GG_+^{(\eta)}$ is locally nilpotent on $w$.
We shall do this by induction on $r$. This is trivial for $r=0$. Now suppose $\GG_+^{(\eta)}$ is locally nilpotent on $w_1=y_2...y_rv$.
There exists $s\in\N$ such that $\prod_{i=1}^{s}$(ad$(x_i))y_1\in [\GG_+ ,\GG_+ ]$ and $\prod_{i=1}^{s}(x_i-\eta(x_i))w_1=0$ for any $x_1,x_2,...,x_s\in\GG_+$. For any $x_1,x_2,...,x_{3s}\in\GG_+$, using these formulas  we have
\begin{equation*}\aligned
&\prod_{i=1}^{3s}(x_i-\eta(x_i))w =  \prod_{i=2s+1}^{3s}(x_i-\eta(x_i))  \prod_{i=1}^{2s}(x_i-\eta(x_i))y_1w_1 \\
 =&\prod_{i=2s+1}^{3s}(x_i-\eta(x_i)) \sum_{p=0}^{2s}\big(\prod_{j=p+1}^{2s}{\text{ad}}(x_{i_j}) y_1\prod_{j=1}^{p}(x_{i_j}-\eta(x_{i_j}))\big)w_1
 \\
 =&\prod_{i=2s+1}^{3s}(x_i-\eta(x_i))\sum_{p=0}^{s-1}\Big(\prod_{j=p+1}^{2s}{\text{ad}}(x_{i_j}) y_1\prod_{j=1}^{p}\big(x_{i_j}-\eta(x_{i_j})\big)\Big)w_1
 \\
 =&\hskip -5pt \prod_{i=2s+1}^{3s}\hskip -5pt (x_i-\eta(x_i))\sum_{p=0}^{s-1}\Big(\big(\prod_{j=p+1}^{2s}\hskip -5pt {\text{ad}}(x_{i_j}) y_1-\eta(\prod_{j=p+1}^{2s}{\text{ad}}(x_{i_j}) y_1)\big)\prod_{j=1}^{p}\hskip -5pt \big(x_{i_j}-\eta(x_{i_j})\big)\Big)w_1\\
  =&0.\endaligned
\end{equation*}
We have used the fact that $\eta(\prod_{j=p+1}^{2s}({\text{ad}}(x_{i_j}) y_1)=0$ since $\prod_{j=p+1}^{2s}({\text{ad}}(x_{i_j}) y_1\in[\GG_+ ,\GG_+ ]$.

Thus,  $\GG_+^{(\eta)}$ is locally nilpotent on $w$.

 (ii)  Let $W$ be a nonzero $\GG$-submodule of $V$ and $w=yv$
is a nonzero vector in $W$ where $y \in U(\GG)$. From (i) we know that there is  $n \in \Z_{>0}$ such that
\begin{equation*}
\prod_{i=1}^n \big(x_i-\eta(x_i)\big)yv =0, {\text{ for any }}  x_1,\ldots,x_n \in \GG_+.
\end{equation*}
 Take $n$ minimal. There exist
$x_2, x_3,\cdots, x_n\in\GG_+$ such that $w'= \prod_{i=2}^n
\big(x_i-\eta(x_i)\big)yv\ne0$ but $\big(x-\eta(x)\big)w'=0$ for all
$x \in \GG_+$. That is $w'$ is a Whittaker vector in $W$.

Part (iii) follows from (ii).\end{proof}

We should notice that the converse of (iii) is in general not true. A counterexample will be presented in  Theorem \ref{thm-critical-case2} (1).  See Remark \ref{infinite-whittaker} for details.

\subsection{ Classical Whittaker modules for $\widehat{\sl_2}$  and  $\widetilde{\sl_2}$: definitions}

Note that for $\widehat{\sl_2}$  and $\widetilde{\sl_2}$,  $(\widehat{\sl_2})_+=(\widetilde{\sl_2})_+$ is generated by $e(0)$ and $f(1)$. Thus the Whittaker function $\eta$ is uniquely determined by $(\lambda,\mu)=(\eta(e(0)),\eta(f(1)))$. A type $\eta$ Whittaker module for $\widehat{\sl_2}$ or $\widetilde{\sl_2}$ is also called a {\it Whittaker module of type} $(\lambda,\mu)=(\eta(e(0)),\eta(f(1)))$.

A Whittaker module of  type $(\lambda,\mu)$ is called non--degenerate  (resp. degenerate) if $\lambda \mu \ne 0$ (resp. $\lambda \mu = 0$).

For any $\lambda,\mu,\kappa\in \C$, let
$J(\lambda,\mu,\kappa )$ be the left ideal of $U(\widehat{\sl_2})$
generated by  \bea \{f(i+1), e(i),f(1)-\mu,e(0)-\lambda, h(i), c-\kappa|i\in
\Z_{>0}\}. \label{gen-ideal} \eea So we have the universal Whittaker module
$V_{\widehat{\sl_2}}(\lambda,\mu,\kappa ):=U(\widehat{\sl_2})/J(\lambda,\mu,\kappa )$ of level $\kappa$
and denote the image of $1$ by $w_{\lambda,\mu,\kappa }$.

Similarly one
may define $V_{\widetilde{\sl_2}}(\lambda,\mu,\kappa ):=U(\widetilde{\sl_2})/\widetilde{J}(\lambda,\mu,\kappa )$,
where $\widetilde{J}(\lambda,\mu,\kappa )$ is the left ideal in $U(\widetilde{\sl_2})$ generated by the set (\ref{gen-ideal}).

It is important to notice that the left ideal $\widetilde{J}(\lambda,\mu,\kappa )$ is ad$(d)$-invariant if and only if $\mu = 0$.

In this paper, we will determine all simple Whittaker modules for
$\widehat{\sl_2}$ and $\widetilde{\sl_2}$, i.e., simple quotients of $V_{\widehat{\sl_2}}(\lambda,\mu,\kappa )$ and $V_{\widetilde{\sl_2}}(\lambda,\mu,\kappa )$.

\begin{remark}
In order to make a difference with imaginary Whittaker modules defined in \cite{Ch1, Ch2},   modules defined in this section will be called {\bf the  classical Whittaker modules.}
\end{remark}


\section{Classical Whittaker modules for  $\widehat{\sl_2}$: Main results}
\label{classical-widehat-sl2}
In this section, according to critical level and non-critical level  we will separately state our precise results on simple quotients of universal Whittaker modules $V_{\widehat{\sl_2}}(\lambda,\mu,\kappa )$ over  $\widehat{\sl_2}$, and we will prove these results in Section \ref{proof-sect-1} after making some preparations in Section \ref{section-Whittaker-b-vir}.

\subsection{Classical Whittaker modules for  $\widehat{\sl_2}$  with noncritical level}
In this subsection we assume that $\kappa\ne-2$.
 Note that $V_{\widehat{\sl_2}}(\lambda,\mu,\kappa )$ is a restricted module. Therefore it is a module over the universal affine vertex algebra $V_{\kappa}(\sl_2)$.
Recall that  every module over $V_{\kappa} (\sl_2)$ for $\kappa\ne -2$ is a module for the Virasoro algebra  generated by the components of the field $L(z) = Y(\omega,z) = \sum_{n \in {\Z} } L(n) z ^{-n-2}$  and that
$$ [L(n), x (m) ] = - m x(n+m) \quad \mbox{for} \ x \in \{ e,f, h \}. $$

 The structure of the module  $V_{\widehat{\sl_2}}(\lambda,\mu,\kappa )$  is described by the following theorem.

\begin{theorem} \label{ired-non-crit} Assume that $\lambda, \mu, \kappa\in\C$ with  $\kappa \ne -2$.
\begin{itemize}
\item[(1)] If $\lambda \cdot \mu \ne 0$, then $V_{\widehat{\sl_2}}(\lambda,\mu,\kappa )$ is a simple $\widehat{\sl_2}$--module.
\item[(2)] If $\lambda \ne 0$, then for any $a\in \C$,  $$M_{\widehat{\sl_2}}(\lambda, 0,\kappa ,a):= V_{\widehat{\sl_2}}(\lambda, 0,\kappa )/U(\widehat{\sl_2})(L(0)-a)w_{\lambda, 0,\kappa }$$ has a unique simple quotient, which we denote by $L_{\widehat{\sl_2}}(\lambda, 0,\kappa ,a)$. Moreover,  any simple quotient of $V_{\widehat{\sl_2}}(\lambda, 0,\kappa )$ is isomorphic to a $ L _{\widehat{\sl_2}}(\lambda, 0,\kappa ,a)$ for some $a\in \C$.
\item[(3)] If $\mu\ne 0$, then for any $a\in \C$,  $$M_{\widehat{\sl_2}}(0 , \mu ,\kappa ,a) :=V_{\widehat{\sl_2}}(0, \mu, \kappa )/U(\widehat{\sl_2})(h(0)/2-L(0)-a)w_{0, \mu,\kappa }$$ has a unique simple quotient, which we denote by $L_{\widehat{\sl_2}}(0,\mu,\kappa ,a)$.  Moreover, any simple quotient of $V_{\widehat{\sl_2}}(0, \mu,\kappa )$ is isomorphic to  $L _{\widehat{\sl_2}}(0,\mu,\kappa ,a)$ for some $a\in \C$.
\item[(4)] The simple quotient of $V_{\widehat{\sl_2}}(0, 0, \kappa )$ is obtained by restricting from the simple highest weight $\widetilde{\sl_2}$-module of level $\kappa$. \end{itemize} \end{theorem}

 The proof of the first Claim is one of the most difficult part of the paper. The proof will be given in Section \ref{proof-sect-1} where we shall prove that $V_{\widehat{\sl_2}}(\lambda,\mu,\kappa )$ is a simple module for the Lie algebra $\widehat{\frak b} \rtimes {\text{Vir}} $. Claims (2)-(4) will be also proved in Section \ref{proof-sect-1} where we need to recall some results on Whittaker modules from \cite{Ch2}.

\subsection{Classical Whittaker modules for   $\widehat{\sl_2}$  at the critical level}

 We will give all simple quotients of  $V_{\widehat{\sl_2}}(\lambda,\mu,-2)$ in this subsection.

Since $V_{\widehat{\sl_2}}(\lambda, \mu, -2)$ is a restricted module for the affine Lie algebra $\widehat{\sl_2}$, it is a module for the universal affine vertex algebra $V_{-2}( \sl_2)$. Recall from Sect.2.1 the element
$ t\in V_{-2}(\sl_2)$ and $T(z) = Y(t,z) = \sum_{n \in {\Z} } T(n) z ^{-n-1}$.
Then $T(n)$ are central elements.

\begin{theorem}  \label{ired-crit}  Let  $c(z)=\sum_{n\le 0} c_n z^{-n-2}, c'(z)=\sum_{n\le 0} c'_n z^{-n-2}\in \C((z))$, and $\lambda,\lambda'\in \C^*,\mu,\mu'\in \C.$
\begin{itemize}
\item[(a)] The Whittaker module $$V_{\widehat{\sl_2}}(\lambda,\mu,-2, c(z))= V_{\widehat{\sl_2} } (\lambda, \mu, -2) / \langle (T(n) - c_n ) w_{\lambda, \mu, -2}, \ n \le 0 \rangle $$ is a simple $\widehat{\sl_2}$ module.
\item[(b)] Any simple quotient of $V_{\widehat{\sl_2}}(\lambda,\mu,-2)$ is of the form $V_{\widehat{\sl_2}}(\lambda,\mu,-2, c(z))$ for some $c(z)$.
\item[(c)] $V_{\widehat{\sl_2}}(\lambda,\mu, -2, c(z))\cong V_{\widehat{\sl_2}}(\lambda',\mu',c'(z))$ if and only if $\lambda=\lambda',\mu=\mu',c(z)=c'(z)$.
\end{itemize}

\end{theorem}

In order to explain the structure of Whittaker modules at the critical level, we shall present two different proofs of the irreducibility of modules $V_{\widehat{\sl_2}}(\lambda,\mu, -2, c(z))$. The first proof will be given is Section \ref{proof-critical-direct} by using similar approach as in the case of non-critical level.  A new insight into the structure of  the Whittaker modules $V_{\widehat{\sl_2}}(\lambda,\mu,-2)$ will be given  in Section \ref{wak-nondeg}, where we shall present their  explicit bosonic realizations and a new  proof of irreducibility.

\begin{remark}
From our proof of irreducibility of $V_{\widehat{\sl_2}}(\lambda,\mu, -2, c(z))$ we will see that $V_{\widehat{\sl_2}}(\lambda,\mu, -2, c(z))$ is simple as a module for the parabolic subalgebra $\widehat{\frak b} ={\frak b} \otimes {\C}[t,t^{-1}] + {\C} c$ of $\widehat{\sl_2}$  where ${\frak b} = {\C} e + {\C} h$.
In fact as a $\widehat{\frak b}$--module $V_{\widehat{\sl_2}}(\lambda,\mu, -2, c(z))$  is isomorphic to the Whittaker module
generated by $v$ such that
$$ e(0) v = \lambda v, \ h(n+1) v = e(n) v =  0  \quad \forall n \ge 1. $$
So our results are analogous to that of W.R. Wallach  \cite{W}  and of D. Mili\v ci\' c and W. Soergel \cite{MS}.
\end{remark}

\section{Whittaker $\widehat{\frak b} \rtimes {\text{Vir}} $--modules}
\label{section-Whittaker-b-vir}
We notice in Proposition \ref{ired-crit-1} that  for  proving irreducibility of any  $V_{\kappa}(\g)$--module, it suffices  to prove that it is simple  as $\LL= \widehat{\frak b} \rtimes {\text{Vir}} $--module.  We will apply this criterion later in the proof of Theorem \ref{ired-non-crit}. In this section we are going to establish the theory of  Whittaker modules over $\LL$.

First we shall define the universal Whittaker $\LL$--module of level $(\kappa_1,\kappa)$. For any $(\lambda,\mu,\kappa_1,\kappa)\in \C^4$, let $$V(\lambda,\mu,\kappa_1,\kappa)=U(\LL)/\langle e(0)-\lambda, L(1)-\mu, L(i+1), h(i),e(i),c_1-\kappa_1,c-\kappa\,|\,i>0 \rangle, $$
 where $w_{\lambda,\mu,\kappa_1,\kappa}$ is the image of $1$.

In order to  prove Theorem \ref{ired-non-crit} (1), we need to prove irreducibility for  $\LL$--modules $V(\lambda, \mu, \kappa, \kappa_1)$ with $\lambda\mu\ne 0$. The main idea is to prove that any Whittaker vector in $V(\lambda,\mu,\kappa_1,\kappa)$ belongs to $\C w_{\lambda,\mu,\kappa_1,\kappa}$, i.e., for any $v\in V(\lambda,\mu,\kappa_1,\kappa)\backslash \C w_{\lambda,\mu,\kappa_1,\kappa}$ we need to find an $x\in \{e(0)-\lambda, L(1)-\mu, L(i+1), h(i),e(i)|i>0\}$ such that $x v\ne 0$.
We  will consider a $\Z$-gradation on $\LL$ different from that with respect to  ${\rm ad}\, L(0)$. This very strange $\Z$-gradation on $\LL$ is crucial to our proof. Now we define it.

 Let $\LL_0=\text{span}\{e(-1),h(0),L(0),c_1,c\}$ and $\LL_i=\text{span}\{e(-i-1),h(-i),L(-i)\}$ for all $i\ne 0$. Then $[\LL_i,\LL_j]\subset \LL_{i+j}$, and $\LL$ and $U(\LL)$ are $\Z$-graded. Denote $D(x)=i$ if $0\ne x\in U(\LL)_i$.

Denote by $\mathbb{M}$ the set of all (infinite) vectors of the form $\mathbf{i}:=(\dots,i_2,i_1)$ with entries
in $\mathbb{Z}_{\ge 0}$, such  that only finitely many entries nonzero. Let $\mathbf{0}=(\dots,0,0)$, $\varepsilon_n=(\ldots,\delta_{l,n},\ldots,\delta_{1,n})$, and $|\mathbf{i}|=\sum_{l=1}^{+\infty} i_l$ for all $\mathbf{i}\in \mathbb{M}$.

For any $\mathbf{i}\in \mathbb{M}$, denote
$$u_{\mathbf{i}}=\cdots (h(-n)^{i_{3n+3}}L(-n)^{i_{3n+2}}e(-n-1)^{i_{3n+1}})\cdots (h(0)^{i_3}L(0)^{i_2}e(-1)^{i_1})\in U(\LL_{\ge 0}).$$

 Let $D(\mathbf{i})=D(u_{\mathbf{i}})=\sum_{j=1}^3\sum_{k=0}^{+\infty} ki_{3k+j}$.


We are going to define another total order on the set $\mathbb{M}$.

Denote by $<$ the {\em reverse lexicographic total order} on $\mathbb{M}$, defined recursively (with respect to
the degree) as follows: $\mathbf{0}$ is the minimum element; and for different nonzero
$\mathbf{i},\mathbf{j}\in \mathbb{M}$  we have $\mathbf{i}<\mathbf{j}$  if and only if one of the
following conditions is satisfied:
\begin{itemize}
\item $\min\{s:i_s\neq 0\}>\min\{s:j_s\neq 0\}$;
\item $\min\{s:i_s\neq 0\}=\min\{s:j_s\neq 0\}=k$ and $\mathbf{i}-\varepsilon_{k}<\mathbf{j}-\varepsilon_{k}$.
\end{itemize}

Define the {\em principal} total order $\prec$ on $\mathbb{M}$ as follows: for different
$\mathbf{i},\mathbf{j}\in \mathbb{M}$ set $\mathbf{i}\prec \mathbf{j}$ if and only if one of the
following conditions is satisfied:
\begin{itemize}
\item $D(\mathbf{i})<D(\mathbf{j})$;
\item $D(\mathbf{i})=D(\mathbf{j})$ and $|\mathbf{i}|<|\mathbf{j}|$;
\item $D(\mathbf{i})=D(\mathbf{j})$ and  $|\mathbf{i}|=|\mathbf{j}|$,
but $\mathbf{i}< \mathbf{j}$.
\end{itemize}

It is clear that
$$B=\{ u_{\mathbf{i}} w_{\lambda,\mu,\kappa_1,\kappa}|\mathbf{i}\in \mathbb{M}\}$$ is a basis of $V(\lambda,\mu,\kappa_1,\kappa)$. Now every element $v$ of $V(\lambda,\mu,\kappa_1,\kappa)$ can be uniquely written in the form
$$v=\sum_{\mathbf{i}\in \mathbb{M}} v_{\mathbf{i}}u_{\mathbf{i}} w_{\lambda,\mu,\kappa_1,\kappa},$$ where only finitely many $v_{\mathbf{i}}\in \C$ are nonzero. We denote by $\supp(v)$ the set of all $\mathbf{i}$  with $v_{\mathbf{i}}\ne 0$.
For  a nonzero $v\in V(\lambda,\mu,\kappa_1,\kappa)$ let $\mathfrak{l}(v)$ denote the maximal (with respect to $\prec$)
element of $\mathrm{supp}(v)$, called the {\em leading term} of $v$. Let $D(v)=D(\mathfrak{l}(v))$, and $D(0)=-\infty$. For any $k\in \Z_{\ge 0}$, set $\supp_k(v)=\{\mathbf{i}\in \supp(v)|D(\mathbf{i})=k\}$.

\begin{lemma}\label{lemma-1} For any $n\in \Z_{> 0}$, let $x=h(n),$ or $ L(n)-\delta_{n,1}\mu,$ or $ e(n-1)-\delta_{n,1}\lambda$. Then, for any  $v\in V(\lambda,\mu,\kappa_1,\kappa)$ with $k=D(v)$ we have
\begin{itemize}
\item[(i)] $D(xv)\le k-n+1$;
\item[(ii)] $\supp_{k-n+1}(xv)\subset \{\mathbf{i}-\mathbf{j}  \ \vert  \ \mathbf{i}\in \supp_k(v), D(\mathbf{j})=n-1\}$.
\end{itemize}
  \end{lemma}

  \begin{proof} We may assume that $v=u_{\mathbf{i}}w_{\lambda,\mu,\kappa_1,\kappa}$ with $k=D(\mathbf{i})$. For any fixed $x=h(n), L(n)-\delta_{n,1}\mu, e(n-1)-\delta_{n,1}\lambda$, by commuting $x$ with some terms of $u_{\mathbf{i}}$ in all possible ways, we may transfer the only negative degree term in $[x,u_{\mathbf{i}}]$ to the right side, i.e., $[x,u_{\mathbf{i}}]\in \sum_{ i\in \{k-n,\ldots,k\}}
   U(\LL_{\ge 0})_{i}\LL_{k-n-i}$. Hence
  \begin{equation}\label{formula-1}xv=[x,u_{\mathbf{i}}]w_{\lambda,\mu,\kappa_1,\kappa}=(u_{k-n}+\sum_{\mathbf{j}\in \{\mathbf{k}|\mathbf{i}-\mathbf{k}\in \mathbb{M}, D(\mathbf{k})=n-1 \}}u_{\mathbf{i}-\mathbf{j}}c_{\mathbf{j}})w_{\lambda,\mu,\kappa_1,\kappa},\end{equation} for some $c_{\mathbf{j}}\in \C$ and $u_{k-n}\in U(\LL_{\ge 0})_{k-n}$.  Hence $D(xu_{\mathbf{i}}w_{\lambda,\mu,\kappa_1,\kappa}) \le  D(u_{\mathbf{i}}w_{\lambda,\mu,\kappa_1,\kappa})-n+1$. The lemma follows.
  \end{proof}

\begin{lemma}\label{lemma-2}Let $\mathbf{i}\in \mathbb{M}$ with $n=\min\{k|i_k\ne 0\}>0$, $\lambda\mu\ne 0$.
\begin{itemize}
\item[(a)] If $n=3k+1$ for some $k\in \Z_{\ge 0}$, then
\begin{itemize}
\item[(i)]  $\mathfrak{l}(h(k+1)u_{\mathbf{i}}w_{\lambda,\mu,\kappa_1,\kappa})=\mathbf{i}-\varepsilon_n$,
\item[(ii)] $\mathbf{i}-\varepsilon_n\not\in \supp(h(k+1)u_{\mathbf{i'}}w_{\lambda,\mu,\kappa_1,\kappa})$ for all $\mathbf{i'}\prec \mathbf{i}$.\end{itemize}
 \item[(b)]If $n=3k+2$ for some $k\in \Z_{\ge 0}$, then
\begin{itemize}
\item[(i)] $\mathfrak{l}((L(k+1)-\delta_{k,0}\mu)u_{\mathbf{i}}w_{\lambda,\mu,\kappa_1,\kappa})=\mathbf{i}-\varepsilon_n$,
\item[(ii)]    $\mathbf{i}-\varepsilon_n\not\in \supp((L(k+1)-\delta_{k,0}\mu)u_{\mathbf{i'}}w_{\lambda,\mu,\kappa_1,\kappa})$ for all $\mathbf{i'}\prec \mathbf{i}$.\end{itemize}
\item[(c)] If $n=3k+3$ for some $k\in \Z_{\ge 0}$, then
\begin{itemize}
\item[(i)]  $\mathfrak{l}((e(k)-\delta_{k,0}\lambda)u_{\mathbf{i}}w_{\lambda,\mu,\kappa_1,\kappa})=\mathbf{i}-\varepsilon_n$,
\item[(ii)]    $\mathbf{i}-\varepsilon_n\not\in \supp((e(k)-\delta_{k,0}\lambda)u_{\mathbf{i'}}w_{\lambda,\mu,\kappa_1,\kappa})$ for all $\mathbf{i'}\prec \mathbf{i}$.
\end{itemize}
\end{itemize}
  \end{lemma}

  \begin{proof}(a) (i) Write $h(k+1)u_{\mathbf{i}}w_{\lambda,\mu,\kappa_1,\kappa}$ as in (\ref{formula-1}). It is clear that the only way to obtain the term $u_{\mathbf{i}-\varepsilon_n}w_{\lambda,\mu,\kappa_1,\kappa}$ is to commute $h(k+1)$ with an $e(-k-1)$, which implies $\mathbf{i}-\varepsilon_n\in \supp(h(k+1)u_{\mathbf{i}}w_{\lambda,\mu,\kappa_1,\kappa})$. Note that $$[h(k+1),L(-k)]w_{\lambda,\mu,\kappa_1,\kappa}=[h(k+1),h(-k)]w_{\lambda,\mu,\kappa_1,\kappa}=0.$$ Combining this with Lemma \ref{lemma-1}, it is easy to see that $\mathfrak{l}(h(k+1)u_{\mathbf{i}}w_{\lambda,\mu,\kappa_1,\kappa})=\mathbf{i}-\varepsilon_n$.

  (ii) Assume that $D(\mathbf{i'})<D(\mathbf{i})$, then from Lemma \ref{lemma-1}, we have $$D(h(k+1)u_{\mathbf{i'}}w_{\lambda,\mu,\kappa_1,\kappa})\le D(\mathbf{i'})-k<D(\mathbf{i}-\varepsilon_n)=D(\mathbf{i})-k.$$ So (ii) holds in this case.
   Assume that $D(\mathbf{i'})=D(\mathbf{i})=s$, and $|\mathbf{i'}|<|\mathbf{i}|$. From Lemma \ref{lemma-1}(2), we know that for any
   $\mathbf{j}\in \supp_{s-k}(h(k+1)u_{\mathbf{i'}}w_{\lambda,\mu,\kappa_1,\kappa}) $
   we have $|\mathbf{j}|\le |\mathbf{i'}|-1<|\mathbf{i}-\varepsilon_n|$. So (ii) also holds in this case.

   Assume that $n'=\min\{k|i_k'\ne 0\}$. If $n'=n$, then from (a), $\mathfrak{l}(h(k+1)u_{\mathbf{i'}}w_{\lambda,\mu,\kappa_1,\kappa})=\mathbf{i'}-\varepsilon_n\prec \mathbf{i}-\varepsilon_n.$ we also have (ii) in this case.

   Now we only need to consider the case $D(\mathbf{i'})=D(\mathbf{i})=s$, and $|\mathbf{i'}|=|\mathbf{i}|$, and $n'>n$. Then from Lemma \ref{lemma-1}, it is easy to see that $D(h(k+1)u_{\mathbf{i'}}w_{\lambda,\mu,\kappa_1,\kappa})<s-k=D(\mathbf{i}-\varepsilon_n)$, which completes the proof of (ii).

   (b) (i) Write $(L(k+1)-\delta_{k,0}\mu)u_{\mathbf{i}}w_{\lambda,\mu,\kappa_1,\kappa})$ as in (\ref{formula-1}). Similarly, the only way to obtained the term $u_{\mathbf{i}-\varepsilon_n}w_{\lambda,\mu,\kappa_1,\kappa}$ is to commute $(L(k+1)-\delta_{k,0}\mu)$ with an $L(-k)$, which implies $\mathbf{i}-\varepsilon_n\in \supp((L(k+1)-\delta_{k,0}\mu)u_{\mathbf{i}}w_{\lambda,\mu,\kappa_1,\kappa}))$. Note that $[L(k+1)-\delta_{k,0}\mu,h(-k)]w_{\lambda,\mu,\kappa_1,\kappa}=0$ and $e(-k-1)$ does not occur in $u_{\mathbf{i}}$. And combining with Lemma \ref{lemma-1}, it is easy to see that to $\mathfrak{l}((L(k+1)-\delta_{k,0}\mu)u_{\mathbf{i}}w_{\lambda,\mu,\kappa_1,\kappa})=\mathbf{i}-\varepsilon_n$.

  (ii)  The proof is similar with (a)(ii). We omit the details.

  Similarly as (a) and (b), we have (c).
  \end{proof}

  From Lemma \ref{lemma-2}, we have
\begin{corollary}\label{cor-1}Suppose that $\lambda\mu\ne 0$. Let $0\ne v\in V(\lambda,\mu,\kappa_1,\kappa)$ with $n=min\{k|\mathfrak{l}(v)_k\ne 0\}>0.$
\begin{itemize}
\item[(i)]  If $n=3k+1$ for some $k\in \Z_{\ge 0}$, then $h(k+1)v\ne 0$.
\item[(ii)] If $n=3k+2$ for some $k\in \Z_{\ge 0}$, then $(L(k+1)-\delta_{k,0}\mu)v\ne 0$.
\item[(iii)]  If $n=3k+3$ for some $k\in \Z_{\ge 0}$, then $(e(k)-\delta_{k,0}\lambda)v\ne 0$.\end{itemize}
 \end{corollary}

  \begin{proposition}\label{prop3.6} If $\lambda\mu\ne 0$, the $\LL$-module $V(\lambda,\mu,\kappa_1,\kappa)$ is  simple.
  \end{proposition}

\begin{proof} From Corollary \ref{cor-1}, any Whittaker vector has to be contained in $\C w_{\lambda,\mu,\kappa_1,\kappa}$. From Lemma \ref{lemma2.1} we see that  $V(\lambda,\mu,\kappa_1,\kappa)$ is a simple $\LL$-module.
\end{proof}

 \section{Proof of Theorems \ref{ired-non-crit}  and \ref{ired-crit} }
\label{proof-sect-1}

Now we are ready to complete the proof of Theorems \ref{ired-non-crit}  and \ref{ired-crit} in this section.

\subsection{Proof of Theorem  \ref{ired-non-crit} (1)}


For any $\mathbf{i},\mathbf{j},\mathbf{k}\in \mathbb{M}$, denote
$$u_{\mathbf{i},\mathbf{j},\mathbf{k}}=(\ldots e(-n)^{i_n}\cdots e(-1)^{i_1})(\cdots h(-n)^{j_{n+1}}\cdots h(0)^{j_1})(\cdots f(-n)^{k_{n+1}}\cdots f(0)^{k_1})$$
in $ U(\widehat{\sl_2}).$
It is clear that
$$B=\{ u_{\mathbf{i},\mathbf{j},\mathbf{k}} w_{\lambda,\mu,\kappa }|\mathbf{i},\mathbf{j},\mathbf{k}\in \mathbb{M}\}$$ is a basis of $V_{\widehat{\sl_2}}(\lambda,\mu,\kappa )$.
For any $\mathbf{i},\mathbf{j},\mathbf{k}\in \mathbb{M}$, denote
$${\mathcal U}_{\mathbf{i},\mathbf{j},\mathbf{k}}=(\ldots e(-n)^{i_n}\cdots e(-1)^{i_1})(\cdots h(-n)^{j_{n+1}}\cdots h(0)^{j_1})(\cdots L(-n)^{k_{n+1}}\cdots L(0)^{k_1}),$$
$\mathcal{B}=\{ {\mathcal U}_{\mathbf{i},\mathbf{j},\mathbf{k}} w_{\lambda,\mu,\kappa }|\mathbf{i},\mathbf{j},\mathbf{k}\in \mathbb{M}\}.$
Then we have the following

\begin{lemma}\label{lemma3.1} Let $\lambda\in \C^*,\mu, \kappa\in \C$ with $\kappa \ne -2$.
Then $\mathcal{B}$ is a basis of the   Whittaker module $V_{\widehat{\sl_2}}(\lambda,\mu,\kappa )$. In particular,  as $\LL$--modules
$$V_{\widehat{\sl_2}}(\lambda, \mu, \kappa ) \cong V(\lambda,\lambda\mu,\frac{3 \kappa}{ \kappa+2},\kappa)$$.
\end{lemma}

\begin{proof} We define a new total order $``>"$ on $\mathbb{M}$ which is different from that defined in Section \ref{section-Whittaker-b-vir}. Recall that $|\mathbf{i}|=\sum_{l=1}^{+\infty} i_l$ for all $\mathbf{i}\in \mathbb{M}$. Let $\mathbf{i}>\mathbf{j}$ if and only if one of the following conditions is satisfied:
\begin{itemize}
\item $|\mathbf{i}|>|\mathbf{j}|$;
\item $|\mathbf{i}|=|\mathbf{j}|$ and there exists some $k_0\in \Z_{>0}$ such that $i_{k_0}>j_{k_0}$ and $i_k=j_k$ for all $k>k_0$.
\end{itemize}

And the total order $``>"$ on $\mathbb{M}^3$ is defined by $(\mathbf{i}, \mathbf{j}, \mathbf{k})>(\mathbf{i'}, \mathbf{j'}, \mathbf{k'})$ if and only if one of the following conditions is satisfied:
\begin{itemize}
\item $\mathbf{k}>\mathbf{k'}$;
\item $\mathbf{k}=\mathbf{k'}$ and $\mathbf{j}>\mathbf{j'}$;
\item $\mathbf{k}=\mathbf{k'}$,$\mathbf{j}=\mathbf{j'}$ and $\mathbf{i}>\mathbf{i'}$ .
\end{itemize}
Now every element $v$ of $V_{\widehat{\sl_2}}(\lambda,\mu,\kappa )$ can be uniquely written in the form
$$v=\sum_{\mathbf{i},\mathbf{j},\mathbf{k}\in \mathbb{M}} v_{\mathbf{i},\mathbf{j},\mathbf{k}}u_{\mathbf{i},\mathbf{j},\mathbf{k}} w_{\lambda,\mu,\kappa },$$ where only finite many $v_{\mathbf{i},\mathbf{j},\mathbf{k}}\in \C$ are nonzero. We denote by $\supp(v)$ the set of all $(\mathbf{i},\mathbf{j},\mathbf{k})$  with $v_{\mathbf{i},\mathbf{j},\mathbf{k}}\ne 0$. For any $0\ne v\in V_{\widehat{\sl_2}}(\lambda,\mu,\kappa )$,  let $\text{deg}(v)$ denote the maximal (with respect to $>$) element of $\mathrm{supp}(v)$,
called the {\em degree} of $v$. For convenience we define deg($0$)=$-\infty$. It is clear that \begin{equation}\text{deg} (u_{\mathbf{i},\mathbf{j},\mathbf{k}}v)=(\mathbf{i},\mathbf{j},\mathbf{k})+\text{deg}(v).\end{equation}

  Let us first prove by induction on $|\mathbf{k}|$ that
  \vskip .2cm
 {\bf Claim.} If $(\mathbf{i'},\mathbf{j'},\mathbf{k'})\in \supp({\mathcal U}_{\mathbf{i},\mathbf{j},\mathbf{k}}w_{\lambda,\mu,\kappa }-(\frac{\lambda}{\kappa+2})^{|\bf{k}|}u_{\bf{i},\bf{j},\bf{k}}w_{\lambda,\mu,\kappa })$, then $\bf{k}>\bf{k'}$.

 \vskip .2cm

It is trivial for $|\mathbf{k}|=0$. Now suppose that $|\mathbf{k}|>0$. Let $j$ be the minimal non-negative integer such that $k_{j+1}\ne 0$.

From the definition of $L(n)$ we have
\begin{equation}L(-j)w_{\lambda,\mu,\kappa }=\frac{1}{\kappa+2}(\lambda f(-j)+\sum_{i=0}^{j-1} e(-j+i) f(-i)+\mu e(-j-1)+u')w_{\lambda,\mu,\kappa },\end{equation} for some  $u'\in U({\hat{h}}^{\le 0})$, where ${\hat{h}}^{\le 0}=\text{span}_{\C} \{h(-n)|n\ge 0\}$. In particular, the Claim also holds for $|{\bf{k}}|=1$. So we only need to consider $|{\bf{k}}|>1$.
Now we may write $${\mathcal U}_{\mathbf{0},\mathbf{0},\mathbf{k}}w_{\lambda,\mu, \kappa}=\frac{1}{\kappa+2}{\mathcal U}_{\mathbf{0},\mathbf{0},\mathbf{k}-\varepsilon_{j+1}} (\lambda f(-j)+\sum_{i=0}^{j-1} e(-j+i) f(-i)+\mu e(-j-1)+u')w_{\lambda,\mu,\kappa },$$ $$v_1=\frac{1}{\kappa+2}(\lambda f(-j)+\sum_{i=0}^{j-1} e(-j+i) f(-i)+\mu e(-j-1)+u'){\mathcal U}_{\mathbf{0}, \mathbf{0},\mathbf{k}-\varepsilon_{j+1}} w_{\lambda,\mu,\kappa },$$
$$v_2=\frac{1}{\kappa+2}[{\mathcal U}_{\mathbf{0}, \mathbf{0},\mathbf{k}-\varepsilon_{j+1}},\lambda f(-j)+\sum_{i=0}^{j-1} e(-j+i) f(-i)+\mu e(-j-1)+u'] w_{\lambda,\mu,\kappa }.$$ Then ${\mathcal U}_{\mathbf{0},\mathbf{0},\mathbf{k}}w_{\lambda,\mu,\kappa }=v_1+v_2$. Note that $v_1, v_2$ depend on $\mathbf{k}$. From induction hypothesis, we have
${\mathcal U}_{\mathbf{0},\mathbf{0},\mathbf{k}-\varepsilon_{j+1}}w_{\lambda,\mu,\kappa }\in$ $$ (\frac{\lambda}{\kappa+2})^{|\bf{k}|-1}u_{\bf{0},\bf{0},\mathbf{k}-\varepsilon_{j+1}}w_{\lambda,\mu,\kappa }
+\sum_{\bf{i''},\bf{j''},\mathbf{k''}\in \mathbb{M}, \mathbf{k}-\varepsilon_{j+1}>\mathbf{k''}}\C u_{\bf{i''},\bf{j''},\mathbf{k''}}w_{\lambda,\mu,\kappa },$$ to give
$$v_1 \in (\frac{\lambda}{\kappa+2})^{|\bf{k}|}u_{\bf{0},\bf{0},\mathbf{k}}w_{\lambda,\mu,\kappa }+\sum_{\bf{i'},\bf{j'},\mathbf{k'}\in \mathbb{M}, \mathbf{k'}<\mathbf{k}}\C u_{\bf{i'},\bf{j'},\mathbf{k'}}w_{\lambda,\mu,\kappa }.$$
By using the fact $[L(n), x(m)]=-mx(n+m)$ for all $x\in \sl_2$, we have
$$v_2\in \sum_{|{\bf{l}}|\le |{\bf{k}}|-2}\,\,\sum_{s,t\in \Z_{\ge 0}} (\C f(-s)+ \C e(-s+1)f(-t)+ U(\hat{h}^{\le 0}){\mathcal U}_{\mathbf{0}, \mathbf{0},\mathbf{l}}w_{\lambda,\mu,\kappa }.$$ Again from induction hypothesis, we have $\deg(v_2)<(\bf{0},\bf{0},\bf{k})$.

So we have proved the claim for $(\bf{i},\bf{j},\bf{k})=(\bf{0},\bf{0},\bf{k})$. Using the fact that $$\deg(u_{\bf{i},\bf{j},\bf{0}}u_{\bf{i}',\bf{j}',\bf{k}})=\deg u_{\bf{i}+\bf{i}',\bf{j}+\bf{j}',\bf{k}},$$ we may easily see that the claim holds for all $(\bf{i},\bf{j},\bf{k})$.

From the Claim  we see that deg$({\mathcal U}_{\mathbf{i},\mathbf{j},\mathbf{k}}w_{\lambda,\mu,\kappa })=(\mathbf{i},\mathbf{j},\mathbf{k})$. Then the linear independence of $\mathcal{B}$ follows. And from the Claim and  by induction on $(\mathbf{k},>)$, we may deduce that $u_{\mathbf{i},\mathbf{j},\mathbf{k}}w_{\lambda,\mu,\kappa }$ is a linear combination of $\mathcal{B}$. This completes the proof of the first statement. Since,
$L(1) w_{\lambda,\mu,\kappa } = \lambda \mu w_{\lambda,\mu,\kappa }, $
 it is easy to see that $V_{\widehat{\sl_2}}(\lambda, \mu, \kappa )$ $\cong V(\lambda,\lambda\mu,\frac{3 \kappa}{ \kappa+2},\kappa)$ as $\widehat{\frak b} \rtimes \text{Vir} $--module. The proof follows.
\end{proof}

 Now    combining this lemma with Propositions  \ref{ired-crit-1} and  \ref{prop3.6} we see  that  $V_{\widehat{\sl_2}}(\lambda, \mu, \kappa )$ is a simple $\widehat{\sl_2}$--module if $\lambda\mu(\kappa+2)\ne0$.

\subsection{Some results from \cite{Ch2}}

We first  recall some notations and results from Chapter 4 in \cite{Ch2}.
Let $$\aligned \Omega=&2(c+2)d+\frac{1}{2}h(0)^2+h(0)+2f(0)e(0) \\ &+2\sum_{n=1}^{+\infty} (e(-n)f(n)+f(-n)e(n)+\frac{1}{2}h(-n)h(n))\endaligned$$ be the Casimir operator for $\widetilde{sl_2}.$

Note that $\Omega=2(\kappa+2)(d+L(0))$ as operators on any restricted $\widetilde{\sl_2}$ module of level $\kappa \ne -2 $.

Let $$z=\frac{1}{2}h(0)^2+h(0)+2f(0)e(0)$$ be the Casimir element for $\mathfrak{S}=\C e(0)+\C h(0)+\C f(0)$. For any $(\dot{d},\dot{z})\in \C^2$, Denote by $${V}_{\widetilde{\sl_2}}(\lambda,0,\kappa ,\dot{d},\dot{z})=\bar{V}_{\widetilde{\sl_2}}(\lambda,0,\kappa )/(U(\widetilde{\sl_2})(d-\dot{d})+U(\widetilde{\sl_2})(z-\dot{z}))w_{\lambda,0,\kappa }.$$

\begin{lemma} \label{lemma3.7}Assume that $\lambda\ne 0$ and $\kappa \ne -2 $. Let $V=V_{\widetilde{\sl_2}}(\lambda,0,\kappa ,\dot{d},\dot{z})$. Then
\begin{itemize}
\item[(1)]  $V$ have a unique simple quotient;
\item[(2)] $\Omega|V=(2(\kappa+2)\dot{d}+\dot{z}){\rm id}_{V}$;
\item[(3)]  If $\dot{z}\ne \frac{(\frac{i}{m}(\kappa+2)-m)^2-1}{2}$ for any $i,m\in \Z_{>0}$, then $V_{\widetilde{\sl_2}}(\lambda,0,\kappa ,\dot{d},\dot{z})$ is simple;
\item[(4)]  Let $M$ be any simple Whittaker module of level $\kappa$ for $\widetilde{\sl_2}$ with a Whittaker vector $w$ of type $(\lambda,0)$.  If there exists some nonzero polynomial $p(x)$ such that $p(d)w=0$, then $M$ is isomorphic to the simple quotient of $V_{\widetilde{\sl_2}}(\lambda,0,\kappa ,\dot{d},\dot{z})$ for some $\dot{d},\dot{z}\in \C$.
\end{itemize}\end{lemma}
\begin{proof} Statements (1)-(4) are the immediately consequences of Proposition 3.15, Corollary 4.4, Corollary 4.16 and Corollary 3.29 in \cite{Ch2}, respectively. We remark that the typo, from the proof of Corollary 4.16 in \cite {Ch2},  $k\in \Z_{\ge 0}$ should be $k\in \Z_{>0}$ there. \end{proof}

\begin{lemma}\label{lemma3.8} Assume that $\kappa \ne -2$. For any $a\in \C$ and any simple Whittaker $\widehat{\sl_2}$-module $V$, we can make $V$ into an $\widetilde{\sl_2}$-module by defining $d =-L(0)+a$ on $V$. Moreover, any simple Whittaker $\widetilde{\sl_2}$-module with central charge $\kappa \ne -2 $ is of this form.\end{lemma}
\begin{proof}The first part of the claim is obvious. The second part follows from the fact the generalized Casimir element $\Omega$ commutes with $\widetilde{\sl_2}$ and acts as a scalar on each simple Whittaker $\widetilde{\sl_2}$ module.\end{proof}

\subsection{Proof of Theorem \ref{ired-non-crit} (2)-(4) }
\label{proof-critical-direct}

(2) Note that $\widehat{\sl_2}$ hence $U(\widehat{\sl_2})$ is $\Z$-graded with respect to $d$.  Denote $U(\widehat{\sl_2})_i=\{u\in U(\widehat{\sl_2})|[d,u]=iu\}$. Then $V_{\widehat{\sl_2}}(\lambda, 0, \kappa )=\oplus_{i\in \Z} U(\widehat{\sl_2})_i w_{\lambda, 0,\kappa }$. Since $\mu=0$,  from Proposition 2.1, Lemma \ref{lemma3.1} and the definition of $V(\lambda, 0, \kappa_1, \kappa)$ in Section 5 we see that
 $U(\widehat{\sl_2})(L(0)-a)w_{\lambda, 0,\kappa }$  is a nonzero proper submodule of  $V_{\widehat{\sl_2}}(\lambda, 0, \kappa )$ for any $a\in \C$.
 Suppose that  $V$ is a maximal submodule of $V_{\widehat{\sl_2}}(\lambda, 0, \kappa )$. Then $V$ must contain a Whittaker vector $v=\sum_{i=0}^r a_iu_{-i}w_{\lambda,0,\kappa }$ where $u_{-i}\in U_{-i}$ and $a_ru_{-r}w_{\lambda,0,\kappa }\ne 0$. It is clear that $u_{-r}w_{\lambda,0,\kappa }$ is also a Whittaker vector (which may be not in $V$).

If $r>0$, from Lemma \ref{lemma3.1}, we know that the image of $u_{-r}w_{\lambda,0,\kappa }$ is a nonzero Whittaker vector in $V_{\widehat{\sl_2}}(\lambda, 0, \kappa )/U(\widehat{\sl_2})(L(0)-a)w_{\lambda, 0,\kappa }$ for all but finitely many $a\in \C$ since $u_{-r}$ has finitely many factors. Therefore the Whittaker  module $V_{\widehat{\sl_2}}(\lambda, 0, \kappa )/U(\widehat{\sl_2})(L(0)-a)w_{\lambda, 0,\kappa }$(which is graded with respect to the action of $L(0)$) is not simple for all but finitely many $a\in \C$. Now we may regarded $V_{\widehat{\sl_2}}(\lambda, 0, \kappa )/U(\widehat{\sl_2})(L(0)-a)w_{\lambda, 0,\kappa }$ as a $\widetilde{\sl_2}$-module by defining $d=-L(0)$, which is of course not simple too for all but finite many $a\in \C$. However, from Lemma \ref{lemma3.7} (3), for any given $\kappa \ne -2 $, we have $$V_{\widehat{\sl_2}}(\lambda, 0, \kappa )/U(\widehat{\sl_2})(L(0)-a)w_{\lambda, 0,\kappa }\cong V_{\widetilde{\sl_2}}(\lambda,0,\kappa ,-a, 2(\kappa+2)a)$$ is simple as $\widetilde{\sl_2}$ module for all but at most countably many $a$, a contradiction. So we have $r=0$.

Now $v\in \C[h(0),L(0)]w_{\lambda, 0,\kappa }$ which was assumed to be a Whittaker vector. From $(e(0)-\lambda)v=0$, we have $v\in \C[L(0)]w_{\lambda, 0,\kappa }$. So $V_{\widehat{\sl_2}}(\lambda, 0, \kappa )/V$ satisfies the condition of Lemma \ref{lemma3.7} (4), which completes the proof.

(3) Let $\sigma$ be the automorphism of $\widetilde{\sl_2}$ defined by $$\sigma(e(i))=f(i+1), \sigma(f(i))=e(i-1),\sigma(h(i))=-h(i)+\delta_{i,0}c,$$ $$ \sigma(c)=c, \sigma(d)=d+\frac{h(0)}{2}.$$ Then $V_{\widehat{\sl_2}}(0,\mu,\kappa )$ is equivalent to $V_{\widehat{\sl_2}}(\mu,0,\kappa )$ via $\sigma$.
 Now (3) follows from (2).

(4) Again any simple $\widehat{\sl_2}$ quotient of $V_{\widehat{\sl_2}}(0,0,\kappa )$ is also a simple $\widetilde{\sl_2}$-module by taking $d=-L(0)$. And (4) follows from Theorem 1.1 in \cite{MZ}.\qed

\subsection{Proof of Theorem \ref{ired-crit}  }

Denote $w=w_{\lambda,\mu,-2}$ in $V_{\widehat{\sl_2}}(\lambda,\mu,-2)$.

\begin{lemma}\label{lemma5-1} Let $\lambda\in \C^*,\mu\in \C, c(z)=\sum_{n\le 0} c_n z^{-n-2}\in \C((z))$.
\begin{itemize}
\item[(a)] The  Whittaker module $V_{\widehat{\sl_2}}(\lambda,\mu,-2)$ has a basis
$\{\mathcal{V}_{\mathbf{i},\mathbf{j},\mathbf{k}} | \mathbf{i},\mathbf{j},\mathbf{k}\in \mathbb{M}\}$
where $\mathcal{V}_{\mathbf{i},\mathbf{j},\mathbf{k}}=$ $$(\ldots e(-n)^{i_n}\cdots e(-1)^{i_1})(\cdots h(-n)^{j_{n+1}}\cdots h(0)^{j_1})(\cdots T(-n)^{k_{n+1}}\cdots T(0)^{k_1}) w.$$
\item[(b)] The Whittaker module $V_{\widehat{\sl_2}}(\lambda,\mu,-2, c(z))=V_{\widehat{\sl_2}}(\lambda,\mu)/\langle T_n-c_n|n\le 0\rangle$ has a basis
 $$\{(\ldots e(-n)^{i_n}\cdots e(-1)^{i_1})(\cdots h(-n)^{j_{n+1}}\cdots h(0)^{j_1})\bar{w}| \mathbf{i},\mathbf{j} \in \mathbb{M}\}$$
 where $\bar{w}$ is the image of $w$.
\item[(c)] $V_{\widehat{\sl_2}}(\lambda,\mu, -2,c(z))$ is simple as $U(\hat{{\frak b}})$-module.
\end{itemize}\end{lemma}
\begin{proof} (a) The proof is similar as that of Lemma \ref{lemma3.1}. We omit the details.

Part (b) follows from (a).

(c) By using completely analogous proof, we may deduce a similar result of Corollary \ref{cor-1} (i), (ii) and (iii) for $\widehat{\frak b}$--module $V_{\widehat{\sl_2}}(\lambda,\mu, -2,c(z))$. Then similar to Proposition \ref{prop3.6}, we can deduce that $V_{\widehat{\sl_2}}(\lambda,\mu, -2, c(z))$ is simple as $\widehat{\frak b}$--module.
\end{proof}
Now we shall finish the proof of Theorem \ref{ired-crit}.
\begin{proof}(a) follows from Lemma \ref{lemma5-1} (c).

(b) Let $\bar{V}$ be any simple quotient of  $V_{\widehat{\sl_2}}(\lambda,\mu, -2)$. From $[T(n),\widehat{\sl}_2]=0$ on $V_{\widehat{\sl_2}}(\lambda,\mu,-2)$,  we have $T(n)$ act as scalar on  $\bar{V}$  for any $n\in \Z$.  Say $\sum_{n\le 0} T(n)z^{-n-2}$ $ =c(z)$ on $\bar{V}$. Then $\bar{V}$ has to isomorphic to $V_{\widehat{\sl_2}}(\lambda,\mu,-2,c(z))$.

(c) It follows from the fact that  $T(z)=\lambda\mu z^{-3}+c(z)$ on $V_{\widehat{\sl_2}}(\lambda,\mu,-2,c(z))$.
\end{proof}


\section{Whittaker modules for  the affine Lie algebra $\widetilde{\sl_2}$}

%
%
\label{section-whittaker-widetilde-sl2}

From Lemma \ref{lemma3.8} and Theorem \ref{ired-non-crit}, we know that simple Whittaker modules for  the affine Lie algebra $\widetilde{\sl_2}$ with noncritical level  are completely determined.
So we only need to consider the critical case, i.e., $\kappa=-2$.

We shall start with one general method for constructing simple $\widetilde{\sl_2}$--modules from simple $\widehat{\sl_2}$--modules:

\begin{theorem}\label{thm-critical-novi}
 Assume that $N$ is any simple $V_{-2}(\sl_2)$--module  (i.e., a simple restricted $\widehat{\sl_2}$ module at the critical level) such that $T(k_0) \ne 0$ on $N$ for certain $k_0 \ne 0$. Then
 $$  \widetilde{N} = {\rm Ind}_{\widehat{\sl_2}}  ^{\widetilde{\sl_2}} N$$
 is a simple ${\widetilde{\sl_2}}$--module.
\end{theorem}
\begin{proof}  Since $N$ is simple, $T(k_0)$ acts on $N$ as a non-zero scalar $c_{k_0}$.
For any $0\ne v\in {\rm Ind}_{\widehat{\sl_2}}^{\widetilde{\sl_2}} N$,  we have $$0\ne (T(k_0)-c_{k_0}  ) ^{i} v \in 1\otimes N $$ for some $i\in \Z_{\ge 0}$. Combining with the simplicity of $N $, we have $U(\widetilde{\sl_2})v={\rm Ind}_{\widehat{\sl_2}}^{\widetilde{\sl_2}} N$. So  $$ {\rm Ind}_{\widehat{\sl_2}}^{\widetilde{\sl_2}} N$$ is simple as $\widetilde{\sl_2}$ module.
\end{proof}

\begin{remark}
Theorem \ref{thm-critical-novi} can be applied on simple modules constructed in \cite{A-2007} and \cite{A-2013}. In this way we get a new family of irreducible ${\widetilde{\sl_2}}$--modules.
\end{remark}

We shall first classify all simple quotients of the universal Whittaker modules at the critical level in the  non-degenerate case.

\begin{theorem}\label{thm-critical}Let $\lambda, \mu\in \C^*$.
 Then any simple quotient of $V_{\widetilde{\sl_2}}(\lambda,\mu,-2)$ is isomorphic to $$ {\rm Ind}_{\widehat{\sl_2}}^{\widetilde{\sl_2}} \Big(V_{\widehat{\sl_2}}(\lambda,\mu,-2, c(z))\Big)$$ for some $c(z)=\sum_{n\le 0} c_n z^{-n-2}$.
\end{theorem}
\begin{proof}      Note that $T(1)$ acts on $V_{\widehat{\sl_2}}(\lambda,\mu,-2)$ as scalar $\lambda\mu\ne 0$.
Irreducibility of the induced module $ {\rm Ind}_{\widehat{\sl_2}}^{\widetilde{\sl_2}} \Big(V_{\widehat{\sl_2}}(\lambda,\mu,-2, c(z))\Big)$ follows from Theorem \ref{thm-critical-novi}.

Now let $W$ be any maximal submodule of $V_{\widetilde{\sl_2}}(\lambda,\mu,-2)$. Note that $T(1)^iT(-i)$, $i\in \Z_{\ge 0}$ commute with $\widetilde{\sl_2}$ as operators on $V_{\widetilde{\sl_2}}(\lambda,\mu,-2)/W$. So for any $i\in \Z_{\ge 0}$, $T(1)^iT(-i)$ act on $V_{\widetilde{\sl_2}}(\lambda,\mu,-2)/W$ as scalars.  Therefore for any $i\in \Z_{\ge  0}$, we have $(T(i)-c_i)w_{\lambda,\mu,-2}\in W$ for some $c_i\in \C$. Now $V$ is a quotient of  $$V_{\widetilde{\sl_2}}(\lambda,\mu,-2)\Big/(\sum_{i\in \Z_{\le 0}}U( \widetilde{\sl_2})(T(i)-c_i)w_{\lambda,\mu,-2})\cong {\rm Ind}_{\widehat{\sl_2}}^{\widetilde{\sl_2}} V_{\widehat{\sl_2}}(\lambda,\mu,-2, c(z)).$$ Thus $V\cong {\rm Ind}_{\widehat{\sl_2}}^{\widetilde{\sl_2}} \Big(V_{\widehat{\sl_2}}(\lambda,\mu,-2, c(z))\Big)$.
\end{proof}

\begin{remark}  \label{remark-proof} Now we have proved all the results in Theorem 1.1. More precisely, Theorem 1.1 (i), (ii), (i') and (ii') follow from Theorem \ref{ired-non-crit} (1), Theorem \ref{ired-crit} (a), Lemma \ref{lemma3.8} and Theorem \ref{thm-critical-novi} respectively. The last statement in Theorem 1.1 follows from Theorem \ref{ired-non-crit} (2), Theorem  \ref{ired-crit} , Lemma \ref{lemma3.8} and Theorem \ref{thm-critical}.\end{remark}

The classification of simple Whittaker modules in the degenerate case at critical level is more complicated. We shall study graded representations of the graded center of $V_{-2}(\sl_2)$. This leads to the study of the infinite-dimensional Lie algebra  $\HH$ from Section   \ref{def-borel-virasoro-critical}.

Next we notice that the universal Whittaker module $V_{ \widetilde{\sl_2}} (\lambda, 0, -2)$ has the structure of right $\HH_{-}$--module (since the ideal $\widetilde{J}(\lambda, 0, -2)$ is $d$--invariant).

Denote by $V_{\widetilde{\sl_2}}(\lambda,0,-2,{\bf 0})$ the quotient of $V_{\widetilde{\sl_2}}(\lambda,0,-2)$ by the submodule generated by $\{dw_{\lambda,0,-2}, T(n)w_{\lambda,0,-2}|n\in \Z\}$.
So $V_{\widetilde{\sl_2}}(\lambda,0,-2,{\bf 0})$ is a   $\hat{\frak b} \rtimes\HH$ -module  on which elements $T(n)$ act trivially.

Recall that $U(\hat{\frak b} \rtimes\HH)$ is the universal enveloping algebra of $\hat{\frak b} \rtimes\HH$. Any $\HH$-module $X$ can endowed with a $\hat{\frak b} \rtimes\HH$-module structure by $\hat{\frak b}X=0$. The resulting $\hat{\frak b} \rtimes\HH$-module will be denoted by $X^{\hat{\frak b} \rtimes\HH}$.

\begin{theorem} \label{classification-quotients}Assume that $\lambda \ne 0$.
 There is one to one correspondence between the equivalence classes of irreducible $\HH_{-}$--modules and simple quotients of $V_{\widetilde{\sl_2}}(\lambda,0,-2)$.
 In particular, as a $\hat{\frak b} \rtimes\HH_{-}$--module, every simple quotient of $V_{\widetilde{\sl_2}}(\lambda,0,-2)$  is isomorphic to the module
 \bea \label{module-1}  X^{\hat{\frak b} \rtimes\HH}  \otimes V_{\widetilde{\sl_2}}(\lambda,0,-2,{\bf 0})\eea
 where $X$ is a simple $\HH_{-}$--module.
 \end{theorem}
Proof of this theorem will be presented in Section \ref{realization-d} where we shall explicitly construct the $\widetilde{\sl_2}$--action on  the $\hat{\frak b} \rtimes\HH_{-}$--module (\ref{module-1}).

Now we shall describe all simple ${\HH}_-$--modules.

 Let ${\frak s}$ be any nonempty subset of $\Z_{<0}$ and $r_{\frak s}\Z$ ($r_{\frak s}>0$) be the additive subgroup of $\Z$ generated by ${\frak s}$. For any map $\chi_{\frak s}: {\frak s}\cup \{0\}\rightarrow \C^*$, we define an associative algebra homomorphism $\phi_{\chi_{\frak s}}:U(\HH)\rightarrow U(\bb)$ by

 $$\phi_{\chi_{\frak s}}(d)=-\frac{r_{\frak s}}{2} h, \phi_{\chi_S}(T(i))=0,\forall i\not\in {\frak s}\cup \{0\}, $$

 $$\phi_{\chi_{\frak s}}(T(i))=\chi_{\frak s}(i)e^{-i/r_{\frak s}},\forall i\in {\frak s}\cup \{0\}.$$

 Then for any $\phi_{\chi_{\frak s}}$ and $\bb$ module $N$, we have a $\HH_{-}$ module $N^{\phi_{\chi_{\frak s}}}=N$ with the action $ xv=\phi_{\chi_{\frak s}}(x)v,\forall x\in {\HH}_{-},v\in N$.
 Recall that all simple modules over $\bb$ are classified in \cite{B}.
\begin{lemma}$N^{\phi_{\chi_{\frak s}}}$ is a simple ${\HH}_{-}$ module if and only if $N$ is a simple $\bb$ module.\end{lemma}
\begin{proof} The necessity is obvious. Now suppose that $N$ is a simple $\bb$ module. If $N$ is finite dimension, then $N$ is 1-dimensional and the claim follows. Now we assume that  $N$ is  an infinite dimensional simple $\bb$ module. Since $eN$ and $\{v\in N|ev=0\}$ are submodules of $N$, we have $e$ acts bijectively on $N$. From the definition it is easy to see that there exists a $m\in \Z_{>0}$ such that $\C[e^{m+i}, h|i\ge 0\}\subset \phi_{\chi_S}(\HH)$. So we only need to prove that $N$ is also simple as  $\C[e^{m+i}, h|i\ge 0]$ module. For any $0\ne v_1, v_2\in N$, we have $e^m v_1\ne 0$ and $v_2=ue^m v_1\in \C[e^{m+i}, h|i\ge 0]v_1$ for some $u\in \C[e,h]$, which completes the proof. \end{proof}

\begin{lemma}\label{b-module}Let $\lambda\ne0$, and $X$ be any simple $\HH$ quotient of $U(\HH)w_{\lambda,0,-2}$. Then  one of the following holds:
\begin{itemize}
\item[(1)] $X$ is 1-dimensional.
\item[(2)] $X\cong N^{\phi_{\chi_{\frak s}}}$ for some ${\frak s}\subseteq \Z_{<0}$ and infinite dimensional simple $\bb$-module $N$.
\end{itemize}
\end{lemma}

\begin{proof}  For any $i\in \Z$, since both $T(i) X$ and  $\{v\in  X| T(i)v=0\}$ are submodules of the simple module $X$, we have $T(i)$ acts on $X$ either bijectively or as zero. Note that $T(0)$ is the central element and $T(i) X=0$ for all $i>0$.  Let ${\frak s}=\{i\in \Z_{<0}| T(i)X\ne 0\}$.  If ${\frak s}=\emptyset$, then $\dim X=1$. So we assume that ${\frak s}\ne \emptyset$. Then $X$ is also a simple $A=\C[T(i),T(i)^{-1}, d| i\in {\frak s}]$ module. Fix an $x=\Pi_{i_k\in {\frak s}_x\subset {\frak s}}T(i_k)^{j_k}\in A$ with $\sum i_kj_k=-r_{\frak s}$, Then for any $i\in {\frak s}$, we have  $T(i)x^{i/r_{\frak s}}\in Z(A)$, which acts on $X$ as a nonzero scalar $a_i$. Now $X$ is a simple module over $A/\langle T(i)x^{i/r_{\frak s}}-a_i|i\in {\frak s}\rangle$. It is clear that we have a homomorphism $$\phi_{\chi_{\frak s}}: A\rightarrow \C[h,e,e^{-1}]$$ $$\phi_{\chi_{\frak s}}(T(i))=a_ie^{-i/r_{\frak s}},\,\,\,\phi_{\chi_{\frak s}}(d)=-\frac{r_{\frak s}}{2} h.$$ Now it is easy to check that $\ker \phi_{\chi_{\frak s}}=\langle T(i)x^{i/r_{\frak s}}-a_i|i\in {\frak s}\rangle$ and we have the induced isomorphism $\bar{\phi}_{\chi_{\frak s}}:A/\langle T(i)x^{i/r_{\frak s}}-a_i|i\in{\frak s}\rangle\rightarrow \C[e,e^{-1},h]$. Now $X$ can be regarded as $\C[e,e^{-1},h]$ module via the isomorphism. Then from definition of $\bb$ structure on $X$,  $X$ is simple as $\phi_{\chi_{\frak s}}(\HH)=\C[e^i,h|-ir_{\frak s}\in {\frak s}]$ module. So $X$ is a simple $\C[e,h]$ module, which gives (2).  \end{proof}

Now we may summarize the main results in this section

\begin{theorem}\label{thm-critical-case2}Let $\lambda\in \C^ *$.
 \begin{itemize}\item[(1)] Any simple quotient $W$  of $V_{\widetilde{\sl_2}}(\lambda,0,-2)$ is isomorphic to $$X^{\hat{\frak b} \rtimes\HH} \otimes V_{\widetilde{\sl_2}}(\lambda,0,-2,{\bf 0}),$$ where $X$ is a simple  $\HH$--module determined in Lemma \ref{b-module}.
     \item[(2)] If $\mu\in \C^*$, then any simple quotient of $V_{\widetilde{\sl_2}}(\lambda,\mu,-2)$ is isomorphic to $$ {\rm Ind}_{\widehat{\sl_2}}^{\widetilde{\sl_2}} \Big(V_{\widehat{\sl_2}}(\lambda,\mu,-2, c(z))\Big)$$ for some $c(z)=\sum_{n\le 0} c_n z^{-n-2}$.\end{itemize}
\end{theorem}

\begin{remark} \label{infinite-whittaker} When $X$ is $1$-dimensional, we know that the action of $d$ is  semisimple on $X^{\hat{\frak b} \rtimes\HH} \otimes V_{\widetilde{\sl_2}}(\lambda,0,-2,{\bf 0}).$ If the action of $d$ is free on $X$, then the action of $d$ is  free on $X^{\hat{\frak b} \rtimes\HH} \otimes V_{\widetilde{\sl_2}}(\lambda,0,-2,{\bf 0})$. This makes a striking difference from the
case of noncritical central charge. In particular, the vector space of all Whittaker
vectors in $X^{\hat{\frak b} \rtimes\HH} \otimes V_{\widetilde{\sl_2}}(\lambda,0,-2,{\bf 0})$ is $X\otimes w_{\lambda,0,-2}$ which can be infinite dimensional.  This shows that the converse of   Lemma  \ref{lemma2.1} (iii) is not true in this case.
 \end{remark}

\begin{remark}
The result for $\lambda=0$ and $\kappa=-2$ in Theorem \ref{thm-critical-case2} may be obtained similarly as  Theorem \ref{ired-non-crit} (3) and (4). We omit the details.\end{remark}

\section{Wakimoto  modules for $\widehat{\sl_2}$}

\label{section-Wakimoto-1}

 In this section we shall review the construction of Wakimoto modules for $\widehat{\sl_2}$ (cf. \cite{W-mod}). Details on the construction of Wakimoto modules using concepts of vertex algebras can be found in \cite{efren}.
 %

Let ${\frak h} = {\C} b $ be the $1$--dimensional commutative   Lie algebra with a symmetric bilinear form defined by $(b,b) =2$, and  $\widehat{\frak h} = h \otimes {\C}[t,t ^{-1}] + {\C} c$ be its affinization. Set $b(n) = b \otimes t ^n $.  Let $\pi^{\kappa+2}$ denote the simple $\widehat{\frak h}$--module of level $\kappa+2$ generated by the  vector ${\bf 1}$ such that
 $$ b(n) {\bf 1} = 0 \quad \forall n \ge 0. $$
 As a vector space
$$\pi^{\kappa+2}  = {\C}[b(n) \,|\,  n \le -1 ], $$

Then $\pi^{\kappa+2}$ has the unique structure of a  vertex algebra  generated  by the field $b(z) = \sum_{n \in {\Z}} b (n) z ^{-n-1} $ such that
$$[b (n), b(m) ] = 2 (\kappa+2) n \delta_{n+m,0} . $$
\vskip 5mm

Let $V_{\kappa}(\sl_2)$ be the universal vertex algebra of level $\kappa$ associated to the affine Lie algebra $\widehat{\sl_2}$. Recall that $M$ is the Weyl vertex algebra from Section 2.4. There is a injective homomorphism of vertex algebras
$\Phi : V_{\kappa}(\sl_2) \rightarrow  M \otimes \pi ^{\kappa+2}$ generated by

\bea e &=& a(-1) {\bf 1}, \nonumber \\  h&=& - 2 a^{*} (0)  a(-1) {\bf 1} + b (-1), \nonumber \\ f &=& - a^* (0) ^2 a(-1) {\bf 1} + k a ^* (-1) {\bf 1} + a ^* (0) b(-1) {\bf 1}. \nonumber \eea

For $x \in \{ e, f, z \} $ we set
$x(z) = \sum_{n \in {\Z}} x(n) z ^{-n-1}. $
We have
\bea
e(z) = Y(e,z)  &=& a(z); \nonumber \\
h(z) = Y(h,z) & =& -2 : a^{*}(z)  a (z) : + b (z); \nonumber \\
              & = & - 2 \left( a(z) ^+ a^* (z)  +  a ^* (z)  a(z) ^- \right) + b(z)  \nonumber \\
f(z) = Y(f, z)  & = & - : a^{*} (z) ^2 a (z) : + k \partial_z a^{*} (z) + a^{*}(z) b (z) \nonumber \\
                &= & - \left( a(z) ^+ (a^* (z) ) ^2 + ( a ^* (z) ) ^2 a(z) ^-  \right) + k \partial_z a^{*} (z) + a^{*}(z) b (z) \nonumber
\eea
The following proposition is a standard result in the theory of vertex algebras (cf. \cite{K}, \cite{LL}, \cite{FB}) applied on the vertex algebra $M \otimes \pi ^{\kappa+2}$.

\begin{proposition} \label{constr}
Assume that $M_1$ is a restricted module for the Weyl algebra  and $N_1$ is a restricted module of level $\kappa+2$  for the Heisenberg algebra $\widehat{\frak h}$,\ i.e., for every $ u \in M_1$ and $v \in N_1$ there is
$N \in {\Zp}$ such that
$$ a(n) u = 0, \quad b(n) v = 0 \qquad \mbox{for} \ n \ge  N.$$
Then $M_1 \otimes N_1$ is a $M \otimes \pi ^{\kappa+2}$--module, and therefore a $V_{\kappa}(\sl_2)$--module.
\end{proposition}

 Assume  that $\kappa \ne -2$.  We have the natural action of the Virasoro algebra generated by  the Sugawara Virasoro vector
\bea \omega &=& \frac{1}{2 (k+2) } (e(-1) f(-1) + f(-1) e(-1) + \frac{1}{2} h(-1) ^2 ) {\bf 1} \nonumber \\ &=&  a(-1) a ^* (-1) + \frac{1}{4  (k+2)} (b(-1) ^2 -2 b(-2) ) \nonumber \\
L(z)  &=&  Y(\omega, z) = \sum_{n \in {\Z} } L(n) z ^{-n-2}.  \nonumber \eea

Assume next  that $\kappa=-2$ (critical level).
The center of $V_{-2}(\sl_2)$ is generated by the field
$$ T(z) = Y( \frac{1}{2 }(e(-1) f(-1) + f(-1) e(-1) + 1/2 h(-1) ^2 ) {\bf 1}, z) = \sum_{n \in {\Z} } T(n) z ^{-n-2}. $$
In our case we have that
$$ T(z) = Y ( \frac{1}{2} ( b(-1) ^2 + 2  b(-2) ){\bf 1}, z) = \frac{1}{2}  (b(z) ^2- 2 \partial_z b(z) ). $$
For details see \cite{FB}.

\section{Whittaker  modules from Wakimoto modules}

\label{whitt-wak}
Because of  Theorem  \ref{ired-non-crit}(1) the structure of   simple  non--degenerate Whittaker modules at non--critical levels is very clear. In this case every  universal Whittaker module is a simple $\widehat{\sl_2}$--module and because of Lemma \ref{lemma3.8}  it is an irreducible Whittaker $\widetilde{\sl_2}$--module. So explicit realization is mostly interesting in the cases of critical level and for degenerate Whittaker modules at non--critical level. In this section we shall see which Whittaker modules can be constructed  by using Wakimoto realization.

 Proposition \ref{constr} gives a very useful method for a  construction of representations for affine Lie algebras of certain level which uses Wakimoto modules. So we just need to construct modules for the vertex algebra $M \otimes \pi ^{\kappa+2}$. Since $V_{\kappa}(\sl_2)$ is a subalgebra of $M \otimes \pi ^{\kappa+2}$,  every $M \otimes \pi ^{\kappa+2}$--module becomes a module for the vertex algebra $V_{\kappa}(\sl_2)$ and therefore for the affine Lie algebra $\widehat{\sl_2}$ of level $\kappa$.
In this section we shall use this fact and construct certain modules of Whittaker type for $\widehat{\sl_2}$. The main new idea in our approach will be in considering certain Whittaker modules for the Weyl algebra as modules for the vertex algebra $M$.

We shall first study a simple case of Whittaker modules for Weyl and Heisenberg algebras.
Let $\lambda, \mu \in {\C}$. Then there is a unique simple module $M_1(\lambda, \mu)$ for the Weyl algebra $W\hskip -2pt eyl$ generated by the vector $v_1$ such that
\bea
&&a(0) v _1= \lambda v_1, \ a(n) v_1 = 0  \quad \forall n \ge 1 \nonumber \\
&& a^{*}(1) v_1 = \mu v_1, \  a ^* (m) v_1 = 0 \quad \forall m \ge 2 \nonumber
\eea
Note that as a vector space $M_1(\lambda, \mu) \cong M$.

Since $M_1(\lambda, \mu)$ is a restricted module for the Weyl algebra we have that $M_1(\lambda, \mu)$ is a simple module over vertex algebra $M$.

Similarly for  $ \chi_0, \chi_1 \in {\C}$ let $N_1(\chi_0, \chi_1)$ be a  module over Heisenberg  algebra $\widehat{\frak h}$ generated by the vector $v_2$ such that
\bea
&& cv_2=(\kappa+2)v_2, \ b(0) v_2 = \chi_0 v_2, \ b(1) v_2 = \chi_1 v_2, \ b(m) v_2 = 0 \quad \forall m \ge 2. \nonumber
\eea
This module is also restricted, and therefore it is a module over the Heisenberg vertex algebra $\pi ^{\kappa+2}$.

So we have $M \otimes \pi ^{\kappa+2}$--module $M_{Wak} (\lambda, \mu, \kappa, \chi_0, \chi_1) := M_1(\lambda, \mu) \otimes N_1( \chi_0, \chi_1)$.

Then $M_{Wak} (\lambda, \mu, \kappa ,  \chi_0, \chi_1)$ is a simple $M \otimes \pi ^{\kappa+2}$--module iff $\kappa \ne -2$.

Let $v=v_1 \otimes v_2$.

\begin{lemma}
We have:
\bea
f(1) v &=& (\chi_1- 2 \mu \lambda) a ^*(0) v + \mu (\chi_0 -\kappa) v + \mu ^2 a(-1)  v , \nonumber \\
f(2) v & = &  \mu ( \chi_1 - \lambda \mu ) v, \nonumber \\
e(0) v &=& \lambda v ,\nonumber \\
h(1) v &=& (\chi_1 - 2 \mu \lambda ) v ,\nonumber \\
e (n) v &=& h(1+n) v = f(2+n) v = 0 \quad \forall n \ge 1. \nonumber
\eea
\end{lemma}
\begin{proof}
We prove this lemma by direct calculation. We have
\bea f(1) v &=& -2 a ^* (1) a ^* (0) a (0) v + a(-1) a ^* (1) ^2 v - \kappa a ^*(1) v + b(0) a ^* (1) v + a ^* (0) b(1) v   \nonumber \\
             &=& (- 2 \mu \lambda + \chi_1) a ^* (0) v +  \mu ( \chi_0 - k) v + \mu ^2 a(-1) v,\nonumber \\
      f(2) v & = &- a^*(1) ^2 a(0) v +  a^*(1) b(1) v = (-\lambda \mu ^2 + \mu \chi_1 ) v = \mu ( \chi_1 - \lambda \mu) v,
             \nonumber  \\ h(1) v &=& -2 a ^* (1)  a (0) v  + b(1) v  = (- 2 \mu \lambda + \chi_1)  v. \nonumber \eea
The proof follows.
\end{proof}

\begin{remark}
  Unfortunately for $\lambda \cdot \mu \ne 0$ , the vector $v$ is not a Whittaker vector and therefore $U(\widehat{\sl_2}) v$ is not a classical Whittaker module. There are some hope that we can construct Whittaker vectors in Wakimoto modules  using methods developed in  the case of Virasoro algebra in \cite{Y}. We will see below that our Wakimoto modules only gives a realization of degenerate classical Whittaker modules of type $(\lambda,0) $ and $(0, \mu)$. Let us note here that in the tensor product modules we can construct classical non-degenerate Whittaker vectors. Let $L_{\widehat{\sl_2} }(\lambda, 0, \kappa_1, a)$ and  $L_{\widehat{\sl_2} }(0, \mu, \kappa_2, b)$ be simple Whittaker modules generated by Whittaker vectors $v_1$ and $v_2$. Then $v_1 \otimes v_2$ is a Whittaker vector of type $(\lambda, \mu)$ and we have
  $$  V_{\widehat{\sl_2} }(\lambda, \mu , \kappa_1+ \kappa_2) \cong U(\widehat{\sl_2} ). (v_1 \otimes v_2) \subset L_{\widehat{\sl_2} }(\lambda, 0, \kappa_1, a) \otimes L_{\widehat{\sl_2} }(0, \mu , \kappa_2, b).$$
   \end{remark}

Our previous lemma is useful for a realization of  degenerate Whittaker module.

\begin{proposition}
Assume that $\mu = 0$ and $\chi_1 = 0 $. Then
$v$ is a Whittaker vector in $M(\lambda, 0 ,  \chi_0, 0)$; i.e.,
$$ e(0) v = \lambda v, \quad f(1) v =0.$$
If $\kappa \ne -2$  then
$$L(0) v = \frac{ \chi_0  (\chi_0 +2 )}{4 (\kappa+2) }  v ,  $$
and there is a non-trivial $\widehat{\sl_2}$--homomorphism
$$ M_{\widehat{\sl_2} } (\lambda, 0, \kappa, \tfrac{ \chi_0  (\chi_0 + 2 )}{4 (\kappa+2) }  ) \rightarrow M_{Wak} ( \lambda, 0, \kappa, \chi_0, 0).$$
\end{proposition}

Now we shall consider the case of critical level.  So let $\kappa =-2$. Then the vertex algebra $\pi^{\kappa +2}$ is commutative.  The irreducible $\pi^0$ modules are one--dimensional. For  $$\chi (z)= \sum _{n \in {\Z} } \chi_n z ^{-n-1}  \in {\C}((z)) $$ let $N_1 (\chi) $ be the $1$--dimensional $\pi^0$ module such that $b(n)$ acts as multiplication by $\chi_n$.

By using  Proposition \ref{constr}  we get a family of $V_{\kappa} (\sl_2)$--modules  realized on the $M \otimes \pi^0$--module
 $$\overline{M_{Wak} } (\lambda, \mu, -2, \chi(z) ):= M_1 (\lambda, \nu) \otimes N_1 (\chi(z)).  $$
 Since $N_1(\chi(z))$ is $1$--dimensional, we have that  $\overline{M_{Wak} } (\lambda, \mu, -2, \chi(z) )$ are actually realized  on the $M$--module $M_1(\lambda, \mu)$ with the following action of $\widehat{sl_2}$:
 \bea
e(z)   &=& a(z); \nonumber \\
h(z) &=& -2 : a^{*}(z)  a (z) : + \chi(z) ; \nonumber \\
f(z)   &=& - : a^{*} (z) ^2 a (z) : -2 \partial_z a^{*} (z) + a^{*}(z) \chi (z) .\nonumber
\eea

\begin{theorem}  \label{ired-wak-1} Let $\lambda,\mu\in\C$ with $\lambda\ne0$. Let $\chi (z) \in {\C}((z))$ be arbitrary.
\begin{itemize}\item[(1)] The $U(\widehat{\sl_2})$--module  $\overline{M_{Wak} } (\lambda, \mu, -2, \chi(z) )$  is simple.
\item[(2)] The $U(\widehat{\sl_2})$--module  $\overline{M_{Wak} } (\lambda, \mu, -2, \chi(z) )$   has the following basis:
$$ e(-n_1-1) \cdots e(-n_r-1) h(-m_1) \cdots h(-m_s) v $$
where  $n_1 \ge \cdots \ge n_r \ge 0$, $m_1 \ge \cdots \ge m_s \ge 0$, $r,s \in {\N}$.\end{itemize}
\end{theorem}
 The proof of this Theorem  will be presented in Section \ref{proof-1}.

In particular, we have obtained explicit realization of simple degenerate Whittaker modules:
\begin{corollary} Assume that
$\lambda\in\C^*$, $\chi (z) \in {\C}((z))$ with    $\chi_1 = 0$. Then the $\widehat{\sl_2}$--module $\overline{M_{Wak} } (\lambda, 0 , -2, \chi(z) )$ is a simple Whittaker module;  i.e.,
$$ V_{\widehat{\sl _2} } (\lambda, 0, -2, c(z) ) \cong \overline{M_{Wak} } (\lambda, \mu, -2, \chi(z) ),  $$
where $$c(z) = \frac{1}{2} (\chi(z)^2  -  2 \partial_z \chi(z) ). $$
\end{corollary}

\begin{remark} As in the case of non-critical levels, our construction of Whittaker modules from the  Wakimoto modules, does not provide a realization of simple non-degenerate Whittaker modules. In Section \ref{wak-nondeg} we shall modify the methods  from this section and present a bosonic realization of modules $V_{\widehat{\sl _2} } (\lambda, \mu , -2, c(z) )$ such that $\lambda, \mu \ne 0$.
\end{remark}



\section{Proof of  Theorem  \ref{ired-wak-1} and more general examples}
\label{proof-1}

In the vertex algebra $(M, Y, \bf{1})$, let $e = a(-1) {\bf 1}$, $\varphi = -2 a (-1) a ^* (0) {\bf 1} \in M$, ${\frak b}_1:= {\C} e + {\C} \varphi \subset M$.
Then ${\frak b}_1$ has the structure of $2$--dimensional complex  Lie algebra with bracket
$$ [\varphi, e] = \varphi_0 e =  2 e.  $$
Let
$\widehat{\frak{b} _1} = {\frak b}_1 \otimes {\C}[t,t^{-1}]+ {\C} c$ be its affinization.

Let $\varphi(z) = Y(\varphi, z) = \sum_{n \in {\Z}} \varphi(n) z ^{-n-1}$. Then
$$ \varphi(n) = -2 \sum_{k \in {\Z}} a^* (k) a (n-k ) \qquad \mbox{for} \ n \ne 0$$
$$ \varphi(0) = -2 \left( \sum_{ k \le -1} a (k) a ^* (-k) +  \sum_{ k \ge 0}  a ^* (-k) a(k) \right). $$

By using commutator formula in the vertex algebra $M$ we get the following relations:
$$ [\varphi(n), \varphi(m)]= - 4 n \delta_{n+m,0}, \ \ [\varphi(n), e (m)] = 2 e (n+m). $$

We shall now prove that the $M$--module $M_1(\lambda, \mu)$ is simple as  $\widehat{\frak{b} _1}$--module of level $\kappa=-2$.

\begin{lemma}
We have:
$$ M_1 (\lambda, \mu)  = U(\widehat{\frak{b}_1 }) v, $$
i.e., $v_1$ is a cyclic vector.
\end{lemma}
\begin{proof}
In order to prove that $v_1$ is a cyclic vector it is enough to verify that
$${\C}[a^* (n) \ \vert \ n \le 0] v \subset U(\widehat{\frak{b}_1 })  v_1. $$
Other basis vector can be constructed by using action of $a(n) = e(n)$, $n \le -1$.

By using the definition of the action of operator $\varphi(0)$ one can easily see that
$$ \mbox{span}_{\C} \{ \varphi(0) ^n v_1 \ \vert \ n \in {\N} \} = \mbox{span}_{\C} \{ a^* (0) ^n v_1 \ \vert n \in {\N} \}.$$
(This is also known from the theory of Whittaker modules for $\sl_2$).

So it remains to prove that

$$ a^* (-n_1) \cdots a^* (-n_r) v \in U(\widehat{\frak{b}_1}) v_1 $$
for all $r \in {\Zp}$, and  $n_1 \ge n_2 \cdots \ge n_r \ge 1.$

By the definition of action of $\varphi(-n)$, $n >0$,  we  get:
$$ \varphi(-n_1) \cdots \varphi(-n_r) v_1 = D  a^* (-n_1) \cdots a^* (-n_r) v_1 + w  \quad (D \ne 0)$$
where $$ w = \sum_{ (\lambda, \mu) \in \mathcal{P} \times \mathcal{P} } C_{\lambda, \mu} u_{\lambda, \mu} v_1 \quad (C_{\lambda, \mu} \in {\C} ) $$
and $\mu_0= (n_1+1, n_2+1, \dots, n_r+1) > \mu$ if  $C_{\lambda, \mu}\ne 0$.
The proof now follows by induction.
\end{proof}

\begin{proposition} \label{ired-1}
For every $\lambda, \mu \in {\C}$ with $\lambda \ne 0$, $M_1(\lambda, \mu)$ is a simple $\widehat{\frak{b} _1}$--module.
\end{proposition}

\begin{proof}
It is enough to prove that every vector $w$ in $ M_1(\lambda, \mu) $ is cyclic.
By applying the action of $a(n)$ we can eliminate basis elements which contain products of $a^* (-n)$.
Then we can assume that $w$ belongs to the space:
$${\C}[ a(-n) \ \vert \ n \ge 1] v_1.$$
So $w$ can be written in the form
\bea \label{exp-1}&& w = \sum_{ \nu \in \mathcal{P} } C_{\nu} u_{\nu,\phi} v_1 \quad (C_\nu \in \C)  \eea
Let $$N=  \mbox{max} \{ \vert \nu  \vert    :\ C_{\nu} \ne 0 \} \qquad  r  = \mbox{min} \{ \ell (\nu  )  : \ C_{\nu} \ne 0, \ \vert \nu \vert = N \}. $$

Take any $\nu_0= (i_1, \dots, i_r )  \in \mathcal{P}$ such that $C_{\nu_0} \ne 0$,  $\ell (\nu_0) = r$, $\vert \nu \vert = N$. (So we choose the shortest possible basis element which appear in (\ref{exp-1}) of maximal size).

 We have
$$ \varphi (i_1) \cdots \varphi  (i_r ) a  (-i_1) \cdots a (-i_r) v_1 = D  a(0) ^r v_1 = D \lambda ^r v_1 \quad (D \ne 0). $$
Moreover if $C_{\nu} \ne 0$ and $\nu \ne \nu_0$ one easily sees that
$$  \varphi (i_1) \cdots \varphi  (i_r ) u_{\nu, \phi} v_1 = 0. $$
Therefore
$$ \varphi (i_1) \cdots \varphi (i_r) w = D \lambda ^r v_1 \quad (D \ne 0). $$ The proof follows.
\end{proof}

Let now $\chi (z)  = \sum_{n \in {\Z} } \chi_n z ^{-n-1} \in {\C}(z)$ is arbitrary. Let $N_1({\chi}) = {\Bbb C}. {1}_{\chi} $ be $1$--dimensional $\pi^0$--module with the property that $b(n)$ acts on  $N_1({\chi})$ by multiplication with $\chi_n$.

Let $h= \varphi + b(-1){\bf 1}$, ${\frak b} = {\C} h + {\C} e$,  $\widehat{\frak b}$ as above. Let $\overline{M_{Wak} } (\lambda, \mu, -2, \chi(z) ) := M_1 (\lambda, \mu)  \otimes {\C} 1_{\chi}$.
As before we set $v= v_1 \otimes 1_{\chi}$.

\begin{lemma}
 Let $\lambda, \mu \in {\C}$, $\lambda \ne 0$, $\chi(z) \in {\C}((z))$. Then   $M_1 (\lambda, \mu)  \otimes {\C} 1_{\chi}$ is a simple  $\widehat{\frak b}$--module.
\end{lemma}
\begin{proof}
In this case, $h(z) =\varphi(z) + \chi(z)$.
One can easily see that irreducibility from the Proposition \ref{ired-1} is not affected if we twist the action of $\varphi(z)$ by $\chi(z)$. The proof follows.
\end{proof}

By using the Wakimoto realization we can extend the action of  $\widehat{\frak b}$ to the action of affine Lie algebra $\widehat{\sl_2}$. Since  $M_1 (\lambda, \mu)  \otimes {\C} 1_{\chi}$  is  a simple $\widehat{\frak b}$--module it is also simple as a module for the affine Lie algebra.   We have:
\begin{theorem}
For every $\lambda, \mu \in {\C}$, $\lambda \ne 0$, $\chi(z) \in {\C}((z))$, $$\overline{M_{Wak} } (\lambda, \mu, -2, \chi(z) ) =M_1 (\lambda, \mu)  \otimes {\C} 1_{\chi} $$  is a simple  $\widehat{\sl_2}$--module at the critical level.
\end{theorem}

Assume that $\chi(z) = \sum_{ k = -p} ^{\infty}  \chi_{-k} z ^{k-1}$, $\chi_p \ne 0$. The case $p=1$ was already discussed above. So let us consider the case $p \ge 2$. We get the conditions of (generalized) Whittaker type
\bea && e(0) v = \lambda v, \label{uv-1}\\
   && h(1) v = ( - 2 \lambda \mu + \chi_1) v,  \quad h(k) v = \chi_k v \ (k=2, \dots, p),  \\
   && f(p+1) v = \mu \chi_p v ,  \\
   && e(n) v = h (n+p ) v = f (n+p+1) v = 0 \quad \forall n \in {\Zp}. \label{uv-4}
   \eea

Modules constructed above are non-isomorphic. This can be proved by using the  action of the center of $V_{-2}(\sl_2)$ which is generated by components of the field $T(z)$.  In our case
$$ T(z) = \frac{1}{2} (\chi(z) ^2 - 2 \partial_z \chi(z) ) \quad \mbox{on} \   \overline{M_{Wak} } (\lambda, \mu, -2, \chi(z) )   .$$

\begin{theorem}
Assume that $\lambda \ne 0$ and  $\chi(z) = \sum_{ k = -p} ^{\infty}  \chi_{-k} z ^{k-1}$, $\chi_p \ne 0$, $p \ge 2$. Then
$$ \overline{M_{Wak} } (\lambda, \mu, -2, \chi(z) )   \cong  \overline{M_{Wak} } (\lambda', \mu', -2, \chi'(z) )  $$ if and only if
$$\lambda = \lambda', \mu = \mu', \chi (z)  = \chi'(z). $$
\end{theorem}
\begin{proof}
Assume that $$ \overline{M_{Wak} } (\lambda, \mu, -2, \chi(z) )   \cong  \overline{M_{Wak} } (\lambda', \mu', -2, \chi'(z) ) . $$ Then conditions (\ref{uv-1})-(\ref{uv-4}) easily imply that
\bea \lambda=\lambda', \ \mu = \mu', \ \chi_k= \chi'_k \ \forall k \ge 1. \label{uv-5} \eea
On the other hand the action of the central elements should be the same on both modules. Therefore
\bea \chi (z) ^2 -  2 \partial_z \chi(z)  = \chi'(z) ^2-   2 \partial_z \chi'(z). \label{uv-6} \eea
By straightforward calculation, using (\ref{uv-5}) and identifying coefficients in the Laurent expansions of (\ref{uv-6}) we get that $\chi(z) = \chi'(z)$. The proof follows.
\end{proof}

Let  $\pi_s$ be the automorphism of $U(\widehat{\sl_2})$ such that
$$ \pi_s (e(n)) = e(n-s), \ \pi_s (f(n) ) = f (n+s), \ \pi_s (h(n) ) = h(n) - s \delta_{n,0}c $$
(cf. \cite{A-2007}). Then for every $s \in {\Z}$, $\pi_s( \overline{M_{Wak} } (\lambda, \mu, -2, \chi(z) ) )$ is also a simple $\widehat{\sl_2}$--module at the critical level generated by the vector $v$ such that
\bea && e(s) v = \lambda v, \nonumber \\
   && h(1) v = ( - 2 \lambda \mu + \chi_1) v,  \quad h(k) v = \chi_k v \ (k=2, \dots, p), \nonumber \\
   && f(p+1-s) v = \mu \chi_p v, \nonumber  \\
   && e(n+s) v = h (n+p ) v = f (n+p+1-s) v = 0 \quad \forall n \in {\Zp}. \nonumber
   \eea

\section{Bosonic Realization of non-degenerate classical Whittaker modules at the critical level}
\label{wak-nondeg}

In this section we shall present a  bosonic realization of the  classical Whittaker modules in the
non-degenerate case. It turns out that we need to extend the  Wakimoto modules to a larger space. The main idea is to replace the  Weyl vertex algebra $M$ with a larger vertex algebra $\Pi(0)$ obtained using the localization of $M$ with respect to $a(-1)$.

Let $L$ be the lattice
$$ L= {\Z} \alpha + {\Z}\beta, \ \la \alpha , \alpha \ra = - \la \beta , \beta \ra = 1, \quad \la \alpha, \beta \ra = 0, $$
and $V_L = M_{\alpha, \beta} (1) \otimes {\C} [L]$ the associated lattice vertex superalgebra, where $M_{\alpha, \beta}(1) $ is the  Heisenberg vertex algebra generated by fields $\alpha(z)$ and $\beta(z)$ and ${\C}[L]$ is the group algebra of $L$.
 We have its subalgebra
$$ \Pi (0) = M_{\alpha, \beta} (1) \otimes {\C} [\Z (\alpha + \beta) ] \subset V_L. $$

The Weyl vertex algebra  $M$ can be realized as a subalgebra of $\Pi(0)$ generated by
$$ a = e^{\alpha + \beta}, \ a^{*} = -\alpha(-1) e^{-\alpha-\beta}. $$
Recall that (cf. \cite{A-2007}, \cite{efren}): $$ M = \mbox{Ker}_{\Pi(0)} e ^{\alpha}_0.$$

This vertex algebra can be described as a localization of Weyl vertex algebra with respect to $a(-1)$, $\Pi(0) = M[(a(-1) ^{-1}]$.
Let us write $a^{-1} :=e ^{-\alpha -\beta}$ and $a^{-1} (n) := e ^{-\alpha-\beta}_{n-2}$. We have the expansion
 $$ Y(a^{-1} , z) = \sum_{n \in {\Z} } a^{-1} (n) z ^{-n +1} . $$
Choose the Virasoro vector
$$\omega = \frac{1}{2} (\alpha (-1) ^2 - \alpha(-2) - \beta(-1) ^2 + \beta(-2) ). $$
Then $\omega$ defines a $\Z$--graduation on $\Pi(0)$ so that $\deg a = 1$ and $\deg a ^{-1} = -1$.
In particular on any $\Pi(0)$--modules we have:
\bea \label{L(0)} [L(0), a (n)] = -n a(n), \ [L(0), a^{-1} (n)] = - n a^{-1} (n).  \label{rel-L0}\eea

The following theorem shows the existence of a Whittaker module for the vertex algebra $\Pi(0)$.

\begin{theorem} \label{ired-pi} Assume that $\lambda \ne 0$.
There is a $\Pi(0)$--module $\Pi_{\lambda}$ generated by the cyclic vector $w_{\lambda}$ such that
$$ a(0) w_{\lambda} = \lambda w_{\lambda},  \quad
a^{-1} (0) w_{\lambda} = \frac{1}{\lambda}  w_{\lambda}, \quad \ a(n) w_{\lambda} =a^{-1} (n) w_{\lambda} = 0 \quad \mbox{for} \ n \ge 1. $$
As a vector space
$$\Pi_{\lambda} \cong {\C} [d(-n), c(-n-1) \  \vert n \ge 0 ] = {\C}[d(0)] \otimes M_{\alpha, \beta} (1), $$
where $c= \alpha + \beta$, $d = \alpha- \beta$.

The module $\Pi_{\lambda}$ is $\Zp$--graded
 $$\Pi_{\lambda} = \bigoplus_{ n \in {\N} } \Pi_{\lambda} (n) $$
 and its lowest component is isomorphic to ${\C}[d(0)]$.
\end{theorem}

\begin{proof}  The proof is based on the construction presented in \cite{BDT} (see also \cite{LW}) . We  shall omit  some of the details.   Let
 $\mathcal{A}$ be  the unital associative algebra generated by generators
 $$ d, \ \ e^{ n c}, n \in {\Z} $$ and relations
 $$[d, e^{ n c} ] = 2 n e^{n c}, \quad e ^{n c} e^{m c} = e^{(n+m) c} \quad (n,m \in {\Z}). $$
 The results from Section 4 of \cite{BDT} implies that for any $\mathcal{A}$-module $U$ and any $\gamma \in \tfrac{1}{2}  \Z d$ there exists  the  unique $\Pi(0)$--module structure on the vector space
$$L_{\gamma} (U) = U  \otimes M_{\alpha, \beta}(1). $$
Moreover, $U  \otimes M_{\alpha, \beta}(1)$ is a simple  $\Pi(0)$--module if and only if $U$ is a simple $\mathcal{A}$--module.
On $L_{\gamma} (U) $ we have
 $$d(0) = d \otimes Id, \ c(0) = \langle c, \gamma \rangle Id. $$

Let $U_{\lambda} $ be the $\mathcal{A}$--module generated by the vector $v_1$ such that $$e^{ n c} v_1 = \lambda ^n  v_1 \quad (n \in {\Z})$$
and that $d$ acts freely.
 Then $U_{\lambda} $ is a simple $\mathcal{A}$--module and $U_{\lambda} = \mathcal{A}. v_1 \cong {\C}[d]$ as a vector space.

 Let  $\gamma = -\frac{1}{2} d $.
 Then
 $$\Pi_{\lambda}:= L_{\gamma} (U_{\lambda}) $$
 is a simple $\Pi(0)$--module.
  Let $v = v_1 \otimes 1$. By construction we have:
 \bea   a(0) v = \lambda v,  \ a^{-1} (0) v = \frac{1}{\lambda} v, \ a(n) v = a^{-1} (n)  v = 0 \quad (n \ge 1). \label{rel-hw-1} \eea
 Since
 $$L(0) v = \frac{1}{2} (c(0) d(0) + d(0) ) v = 0$$relations (\ref{rel-L0}) and (\ref{rel-hw-1}) imply that $L(0)$ acts semisimply on $\Pi_{\lambda}$ and it defines a required $\Zp$--graduation.
   The proof follows.
\end{proof}

\begin{remark}
Note that $\Pi(0)$ is $\Z$--graded, while it is not $\Zp$--graded. But one can also apply Zhu's algebra theory to construct and classify its $\Zp$--graded modules.
One can show that in our case Zhu's algebra $A(\Pi(0))$ is isomorphic to the associative algebra $\mathcal{A}$  and therefore  $U$ has the structure of a simple module for $A(\Pi(0))$.  So  the existence of the module $\Pi_{\lambda}$ also follows from  Zhu's algebra theory.
\end{remark}

Since $M \subset \Pi(0)$ we can consider $\Pi_{\lambda}$ as a $\widehat{\frak b}_1$--module.

\begin{lemma} \label{ired-7} Let $\lambda\ne0$. Then we have:\begin{itemize}
\item[(i)]
$\Pi_{\lambda}  \cong M_1(\lambda, 0) $  as  $M$--modules;
\item[(ii)]  $\Pi_{\lambda} $ is  simple $\widehat{\frak b}_1$--module.\end{itemize}
\end{lemma}

\begin{proof}
  We shall first consider $\Pi_{\lambda}$ as a module over the Weyl vertex algebra $M$. We directly see that
 $$ a(0) w_{\lambda} = \lambda w_{\lambda}, \ a(n+1) w_{\lambda} = a^* (n+1) w_{\lambda} = 0 \quad (n \ge 0). $$
 Therefore the cyclic $M$--submodule $M. w_{\lambda} $  of $\Pi_{\lambda}$ is isomorphic to $M_1(\lambda, 0 )$.

 On the other we have relation
 $$ a(z) a^{-1} (z) = Y(a (-1)  a^{-1} {\bf 1}, z)  =Y ({\bf 1},z) = Id $$
 which easily implies that
 $${\C} [a^{-1} (-n) \ \vert \ n \ge 0]  w_{\lambda} \subset {\C} [a (-n) \ \vert \ n \ge 0]  w_{\lambda } \subset M. w_{\lambda}  . $$
 This proves that  $ M. w_{\lambda} = \Pi_{\lambda} $, and assertion (i) holds. Assertion (ii) follows from (i) and
Proposition \ref{ired-1}.
\end{proof}

Let $M_T(0)$ be the commutative vertex subalgebra of $V_{-2}(\sl_2)$ generated by $T(z)=\sum_{n \in {\Z} } T(n) z ^{-n-2}$.

By using standard calculations in vertex algebras we get:
\begin{proposition} \label{embedding-new}
There is an embedding of vertex algebras
$$ V_{-2} (\sl_2) \rightarrow M_T (0) \otimes \Pi(0)$$
such that
\bea
e & = &  a  , \label{def-e-2} \\
h & = & - 2 \beta(-1) = -2 a^* (0) a(-1) {\bf 1}  \label{def-h-2} ,\\
f & = & \left[  T(-2) - (\alpha(-1) ^2 - \alpha(-2) ) \right] a^{-1} \label{def-f-2} \\
  &=& - a^* (0) ^2 a(-1) {\bf 1} -2  a ^* (-1) {\bf 1} +  T(-2) a^{-1}.
\eea
\end{proposition}
For any $\chi(z) = \sum_{n \in {\Z} } \chi(n)  z^{-n-2} \in {\C}((z)) $ let $M_T(\chi(z))$
be the $1$-dimensional $M_T(0)$--module such that $T(n)$ acts as multiplication with $\chi(n) \in {\C}$.

We have:

\begin{theorem} \label{wak-nondeg-tm} Let $\lambda  \ne 0$.
Let $$ \chi (z) = \frac{\lambda \mu }{ z^3} +  c(z), \quad c(z)= \sum_{n \le  0 } \chi (n) z ^{-n-2} \in {\C}((z)). $$
Then
$$V_{\widehat{\sl_2} } (\lambda, \mu, -2, c(z) ) \cong M_T( \chi(z) ) \otimes \Pi_{\lambda }. $$
\end{theorem}

\begin{proof}
 Irreducibility of  $M_T( \chi(z) ) \otimes \Pi_{\lambda }$  as a $\widehat{\frak b}_1$--module follows from Lemma \ref{ired-7}. Therefore $M_T( \chi(z) ) \otimes \Pi_{\lambda }$ is a simple $\widehat{\sl_2}$--modules at the critical level on which $T(z)$ acts as $\chi(z)$. It remains to prove that
 $M_T( \chi(z) ) \otimes \Pi_{\lambda  }$ is generated by Whittaker vector of type $(\lambda, \mu)$.

  By construction we have
 \bea
 e(0) (1 \otimes w_{\lambda} ) &=& 1 \otimes a(0) w_{\lambda } = \lambda (1 \otimes w_{\lambda} ), \nonumber \\
 f(1) (1 \otimes w_{\lambda } & = & T(1).1  \otimes a^{-1} (0)  w_{\lambda } = \mu  (1 \otimes w_{\lambda} ), \nonumber \\
 e(n+1)  (1 \otimes w_{\lambda} ) &=&  h(n+1)  (1 \otimes w_{\lambda} ) =  f(n+2)  (1 \otimes w_{\lambda} ) =0\quad (n \ge 0). \nonumber
 \eea
 Therefore $1 \otimes w_{\lambda}$ is a Whittaker vector of type $(\lambda, \mu)$. The proof follows.\end{proof}

\section{Realization of   $\widetilde{\sl_2}$--modules and proof of Theorem \ref{classification-quotients}}
\label{realization-d}

In this section we shall see that all irreducible degenerate Whittaker   $\widetilde{\sl_2}$--modules at the critical level can be also constructed by using explicit realization from
Section \ref{wak-nondeg}.  In particular we will show that for every $\HH$--module $X$, the tensor product $X \otimes \Pi_{\lambda}$ becomes a $\widetilde{\sl_2}$--module.  We shall use this construction to prove  Theorem \ref{classification-quotients} on classification of irreducible degenerate Whittaker modules at the critical level.

First we notice that if we set $d =- L(0)$, then the module $\Pi_{\lambda}$ from Section \ref{wak-nondeg} becomes a $\widetilde{\sl_2}$--module. Moreover, for every $\HH$--module $X$, the tensor product $X \otimes \Pi_{\lambda}$ becomes a $\widetilde{\sl_2}$--module.  More precisely we have the following result.

\begin{theorem} \label{ired-constr-1} Assume that $\lambda \ne 0$. Then the following holds.
\begin{itemize} \item[(1)]  $\Pi_{\lambda}\cong V_{\widetilde{sl_2}} (\lambda, 0, -2, {\bf 0} )$  as $\widetilde{\sl_2}$--modules.
 \item[(2)] Assume that $X$ is a simple, restricted  $\HH$--module. Then $X\otimes \Pi_{\lambda}$ is a simple $\widetilde{\sl_2}$--module at the critical level where the action of the degree operator is given by
 $${\bf d} =  d\otimes \mbox{Id} - \mbox{Id} \otimes L(0). $$
\end{itemize}
  \end{theorem}
\begin{proof}
First we notice that $$\Pi_{\lambda}\cong V_{\widetilde{sl_2}} (\lambda, 0, -2, {\bf 0} )$$ as $\widehat{sl_2}$--modules. By using relation (\ref{L(0)}) and the realizations in Proposition \ref{embedding-new} we see that  on $\Pi_{\lambda}$  we have
$$ L(0) v_{\lambda} = 0,  \ [L(0), x(m)] = -m x(m), \quad x \in \{e, f, h\}. $$
This shows that if we set $d= -L(0)$,  $\Pi_{\lambda}$  is also a $\widetilde{\sl_2}$--module and that the action of $d$ is compatible with the $\widehat{sl_2}$--isomorphism above. This proves the assertion (1).

Let us prove (2). Since $X$ is a restricted $\HH$--module, it is also a module for the vertex algebra $M_T(0)$. This implies that  $X\otimes \Pi_{\lambda}$ is a $V_{-2}(\sl_2)$--module, and therefore a $\widehat{\sl_2}$--module at the critical level. By using the fact that $X$ is a $\HH$--module, the realization from Proposition \ref{embedding-new} will again imply that
$$ [ {\bf d}, x(n)] = n x (n), \quad x \in \{e, f, h\}. $$
This shows that $X\otimes \Pi_{\lambda}$ is a  $\widetilde{\sl_2}$--module. The irreducibility result follows from the irreducibility of $\Pi_{\lambda}$ as $\widehat{\frak b}$--module and Proposition \ref{constr-bt-modules}.
\end{proof}

Next result  will imply the proof of  Theorem \ref{classification-quotients}.

\begin{proposition} Assume that $\lambda \ne 0$. Then the following holds.
\begin{itemize}
 \item[(1)] $U(\HH_-) \otimes  \Pi_{\lambda}\cong V_{\widetilde{sl_2}} (\lambda, 0, -2)$  as $\widetilde{\sl_2}$--modules.
 \item[(2)] If $X$ is a simple $\HH_{-}$--module, then $X \otimes \Pi_{\lambda}$ is isomorphic a simple quotient of $V_{\widetilde{sl_2}} (\lambda, 0, -2)$.
 \item[(3)] Any simple quotient of $ V_{\widetilde{sl_2}} (\lambda, 0, -2)$ is isomorphic to a module $X\otimes \Pi_{\lambda}$ for certain simple $\HH_{-}$--module $X$.
\end{itemize}
 In particular, there is one to one correspondence between the equivalence classes of simple $\HH_{-}$--modules and simple quotients of $V_{\widetilde{\sl_2}}(\lambda,0,-2)$.
 \end{proposition}

\begin{proof}(1)
From Lemma \ref{lemma5-1}(a), $V_{\widetilde{\sl_2}}(\lambda,0,-2)$ has a basis:
$$(\ldots e(-n)^{i_n}\cdots e(-1)^{i_1})(\cdots h(-n)^{j_{n+1}}\cdots h(0)^{j_1})d^{i}(\cdots T(-n)^{k_{n+1}}\cdots T(0)^{k_1}) w $$ for all
 $i\in \Z_{\ge 0}$, $\mathbf{i},\mathbf{j},\mathbf{k}\in \mathbb{M}$, where $w= w_{\lambda,0,-2}$.

By using explicit realization from previous section we see  that
\bea  (\hskip -1pt \ldots e(\hskip -4pt -n)^{i_n}\cdots e(\hskip -4pt -1)^{i_1})(\cdots h(\hskip -4pt -n)^{j_{n+1}}\cdots h(0)^{j_1})d^{i}(\cdots T(-n)^{k_{n+1}}\cdots T(0)^{k_1}) (1 \otimes v_{\lambda} )\nonumber \\ =\hskip -4pt (d^{i}(\cdots T(\hskip -2pt -n)^{k_{n+1}}\hskip -2pt \cdots \hskip -2pt T(0)^{k_1}) )  \otimes (\hskip -2pt \cdots \hskip -2pt e(-n)^{i_n}\hskip -2pt \cdots \hskip -2pt e(-1)^{i_1})(\hskip -2pt \cdots\hskip -2pt  h(-n)^{j_{n+1}}\hskip -2pt \cdots \hskip -2pt h(0)^{j_1}) v_{\lambda}. \nonumber
\eea
 This relation gives an isomorphism $\Psi :  V_{\widetilde{sl_2}} (\lambda, 0, -2) \rightarrow U(\HH_-) \otimes  \Pi_{\lambda} $ such that $\Psi(w) = 1 \otimes v_{\lambda}$.

(2) First we notice that $X  $ can be considered as a simple $\HH$--module by letting $T(n)$ acts as zero on $X$ for $n>0$. From Theorem \ref{ired-constr-1} (1) now it follows that $X \otimes \Pi_{\lambda}$ is a simple  $\widetilde{\sl_2}$--module at the critical level. Take an arbitrary non--trivial vector $\overline{w} \in X$. By construction,   $\overline{w} \ \otimes v_{\lambda}$ is a Whittaker, cyclic vector of type $(\lambda, 0)$. By using the universal property of the Whittaker module  $ V_{\widetilde{\sl_2} } (\lambda, 0, -2) $  we conclude that $X\otimes \Pi_{\lambda}$ is isomorphic to a simple quotient of $V_{\widetilde{\sl_2} } (\lambda, 0, -2) $ .

(3)  By using (1) and    Proposition \ref{constr-bt-modules}
we see that every simple quotient  of $V_{\widetilde{\sl_2} } (\lambda, 0, -2) $ is isomorphic to a module $X \otimes \Pi_{\lambda}$ for certain simple $\HH_{-}$--module $X$. This proves the assertion (3).

The last statement follows directly from (2) and (3).
 \end{proof}

The  results (or part) in this paper were presented on several conferences, including,  Lie Algebra Conference at Harish-Chandra Institute in December of 2014,  Representation Theory Workshop at
Uppsala in May of 2015, Representation Theory  XIV at
Dubrovnik in June of 2015, the 14th National Conference on Lie algebra at Xinyang, China in July of 2015, Vertex Operator Algebra and Related Topics at Chengdu in September of 2015.

\vskip 5mm
\noindent {\bf Acknowledgments.} Part of the research presented in this paper was carried out during the visit of the
first two authors to  Wilfrid Laurier University in April of 2014 and of the second author to University of Waterloo in 2014.  The second author
thanks professors Wentang Kuo and Kaiming Zhao for sponsoring his visit, and University of Waterloo
for providing excellent working conditions. The first author thanks Kaiming Zhao and Wilfrid Laurier University for hospitality during his visit. At  last we sincerely thank the referee for   nice suggestions to make the paper more readable.

D.A. is partially supported by the Croatian Science Foundation under the project 2634.

R.L. is partially supported by NSF of China (Grants 11471233, 11371134) and a Project Funded by the Priority Academic Program Development of Jiangsu Higher Education Institutions.

K.Z. is partially supported by  NSF of China (Grants 11271109, 11471233), NSERC (Grant 311907-2015) and University Research Professor Grant at WLU(244382).

\

\end{document}